\documentclass[reqno,12pt,final]{amsart}
\usepackage{eurosym}
\usepackage{amsmath,amsthm,amsfonts,amssymb}
\usepackage[english]{babel}
\usepackage{graphicx}
\usepackage{enumerate}

\newtheorem{theorem}{Theorem}
\newtheorem{lemma}{Lemma}
\newtheorem{remark}{Remark}

\allowdisplaybreaks
\input{tcilatex}

\begin{document}
\title[New cost terms through the homogenization]{ \textbf{New cost terms
through the homogenization of an optimal control problem under dynamic
boundary conditions on the microscopic particles }}
\author{J.I. D\'{\i}az$^{1}$}
\address{$^{1}$Instituto de Mathematica Interdisciplinar, Universidad
Complutense de Madrid, Spain.\\
\texttt{jidiaz@ucm.es}}
\author{T.A. Shaposhnikova$^{2}$}
\address{$^{2}$Lomonosov Moscow State University, Russia.\\
\texttt{shaposh.tan@mail.ru, avpodolskiy@yandex.ru}}
\author{A.V. Podolskiy$^{2}$}
\maketitle

\textbf{Abstract }Given an optimal control problem on a heterogeneous\textbf{%
\ }body with\textbf{\ }a periodical structure of particles depending on a
small parameter $\varepsilon $\textbf{, }we study the asymptotic behavior,
as $\varepsilon \rightarrow 0,$ of the optimal control functional and the
optimal state when the initial problem is of parabolic type, and when on the
particles' boundary, we assume a dynamic condition and the actuation of some
controls for some subset of the particles. We show, in the so-called
\textquotedblleft critical case\textquotedblright \ (concerning a certain
relation between the structure's period, the diameter of the balls, and the
growth coefficient of the particles boundary condition), the appearance of
some new non-local in time "strange terms", not only in the limit parabolic
equation but also in the limit cost functional. Microscopic localized
controls generate peculiar terms in both the limit equation and the cost
functional that do not appear in the case of controls applied to the entire
set of particles or when the boundary condition on the particles is of Robin
type.


\section{Introduction}

The problem we will consider arises in many different fields. For instance,
it is well-known that many problems in Chemical Engineering lead to the
optimization of some cost functionals (\cite{Upreti book}, \cite%
{Nolasco-Survey}) and the same happens in the Porous Media Theory in which
the word \textquotedblleft particle\textquotedblright \ must be replaced by
\textquotedblleft perforation\textquotedblright \ (see, e.g. \cite{Hurlov}, 
\cite{Hornung-Yaeger}, \cite{ConcaDliTi}, \cite{OlSh95}, \cite{GoSanchezP}, 
\cite{Iliev-Mikelik}, \cite{Timof}, \cite{Anguiano} and the many other
references quoted in the monographs \cite{DiGoShBook}). Simplified models in
Climatology can be also modeled in terms very close to the ones we will
study in this paper (\cite{DiPoShRACSAM}).

The main goal of this paper is to illustrate how the homogenization of some
optimal control problems may give rise to new non-local in time "strange
terms", not only in the limit parabolic equation but also in the limit cost
functional, assuming a dynamic boundary condition and the actuation of some
controls on some subset of the particles. We will do that for the so-called
\textquotedblleft critical case\textquotedblright ,\ that is characterized
by certain relation between the structure's period, the diameter of the
balls, and the growth coefficient in the particles' boundary condition. In
this way, microscopic localized controls generate peculiar terms in both the
limit equation and the cost function that do not appear, for instance, in
the case of the Robin type boundary condition on the particles. 

\bigskip We give a detailed presentation of the heterogeneous domain $\Omega
_{\varepsilon }$ in the next Section, but for the moment, we outline that,
since in the most of the\textrm{\ }cases it is impossible to act over the
entire spatial domain $\Omega _{\varepsilon }$, the control is applied only
on the boundary of the particles contained in a small portion of the domain (%
$\omega $ such that $\overline{\omega }\subset \Omega $). Thus, the set of
boundaries of the internal particles is constituted in the form $%
S_{\varepsilon }=S_{\varepsilon }^{1}\bigcup S_{\varepsilon }^{2}$, where $%
S_{\varepsilon }^{2}$ is the set of boundaries of the controlling particles $%
G_{\varepsilon }^{2}$ and $S_{\varepsilon }^{1}$ is the set of boundaries of
the particles $G_{\varepsilon }^{1}$ to which no control is implemented. The
state of the control problem is given through 
\begin{equation}
\left \{ 
\begin{array}{lr}
\partial _{t}u_{\varepsilon }(v)-\Delta u_{\varepsilon }(v)=f, & (x,t)\in
Q_{\varepsilon }^{T}, \\ 
\varepsilon ^{-\gamma }\partial _{t}u_{\varepsilon }(v)+\partial _{\nu
}u_{\varepsilon }(v)=\varepsilon ^{-\gamma }\chi _{S_{\varepsilon }^{2,T}}v,
& (x,t)\in S_{\varepsilon }^{T}, \\ 
u_{\varepsilon }(v)(x,0)=0, & x\in \Omega _{\varepsilon }\bigcup
S_{\varepsilon }, \\ 
u_{\varepsilon }(v)(x,t)=0, & (x,t)\in \Gamma ^{T},%
\end{array}%
\right.   \label{init state prob}
\end{equation}%
where $f\in L^{2}(Q^{T})$ and\textrm{\ }$v\in L^{2}(S_{\varepsilon }^{2,T})$
is the control.\textrm{\ }Here, we are using the notation (considering $%
0<T<\infty $)%
\begin{equation}
\begin{array}{llll}
\Omega _{\varepsilon }=\Omega \setminus \overline{G_{\varepsilon }}, & 
S_{\varepsilon }=\partial G_{\varepsilon }, & \partial \Omega _{\varepsilon
}=S_{\varepsilon }\cup \partial \Omega , &  \\ 
Q_{\varepsilon }^{T}=\Omega _{\varepsilon }\times (0,T), & \Gamma
^{T}=\partial \Omega \times (0,T), & S_{\varepsilon }^{T}=S_{\varepsilon
}\times (0,T), & Q^{T}=\Omega \times (0,T),%
\end{array}
\label{def: basic sets}
\end{equation}%
which will be described in detail in the next section. We note that $%
G_{\varepsilon }$ is the set of small particles ( $\varepsilon $%
-periodically distributed and homothetic to a unit ball $G_{0}$) in an open
bounded regular set $\Omega $ of $\mathbb{R}^{n}$, $n\geq 3$. By $\chi
_{S_{\varepsilon }^{2,T}}$, we denote the characteristic function of the set 
$S_{\varepsilon }^{2,T}=S_{\varepsilon }^{2}\times (0,T)$ that lies entirely
in the set $\omega _{\varepsilon }^{T}$ defined below 
\begin{equation*}
\omega _{\varepsilon }=\omega \cap \Omega _{\varepsilon },\quad \omega
_{\varepsilon }^{T}=\omega _{\varepsilon }\times (0,T),\quad \omega
^{T}=\omega \times (0,T).
\end{equation*}%
\ The parameter $\gamma >0$ plays a crucial role since in this paper we
consider the so-called \textquotedblleft critical case\textquotedblright \
governed by the size of particles that are translations of a small particle $%
a_{\varepsilon }G_{0},$ where $G_{0}$ is the unit ball with radious $%
a_{\varepsilon }=C_{0}\varepsilon ^{\gamma }$, $\gamma =\frac{n}{n-2},$ and $%
C_{0}$ is some positive constant.

\begin{figure}[!t]
\textrm{\center
\includegraphics[width=0.5%
\linewidth]{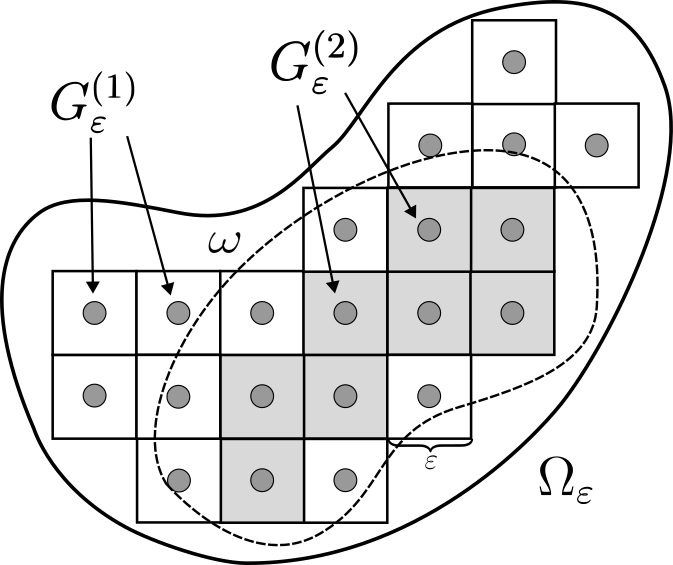} }
\caption{Perforated domain. The control is implemented only on $S_{\protect%
\varepsilon }^{2}$, the boundary of some of the internal balls: the ones
collected under the $G^{2}_\protect \varepsilon$.}
\label{fig: perforated domain}
\end{figure}


Notice that since the problem (\ref{init state prob}) is linear, by some
obvious change of variable, we can also consider the case of a non-zero
initial datum. To finalize the statement of the optimal control problem, we
introduce the cost functional $J_{\varepsilon
}:L^{2}(0,T;L^{2}(S_{\varepsilon }^{2}))\rightarrow \mathbb{R}$ 
\begin{equation}
J_{\varepsilon }(v)=\frac{1}{2}\Vert \nabla {u_{\varepsilon }(v)}\Vert
_{L^{2}(Q_{\varepsilon }^{T})}^{2}+\frac{1}{2}\int \limits_{\Omega
_{\varepsilon }}u_{\varepsilon }^{2}(v)(x,T)dx \\
+\frac{{\varepsilon }^{-\gamma }}{2}\int \limits_{S_{\varepsilon
}}u_{\varepsilon }^{2}(v)(x,T){ds}+\varepsilon ^{-\gamma }\frac{N}{2}\Vert {v%
}\Vert _{L^{2}(S_{\varepsilon }^{2,T})}^{2},  \label{cost functional}
\end{equation}%
where $N>0$. Then, the optimal control $v_{\varepsilon }$ is given by 
\begin{equation}
J_{\varepsilon }(v_{\varepsilon })=\inf \limits_{v\in
L^{2}(0,T;L^{2}(S_{\varepsilon }^{2}))}J_{\varepsilon }(v).
\label{optim control def}
\end{equation}%
In what follow, we will abuse the notation and simply write $u_{\varepsilon
} $ instead of $u_{\varepsilon }(v_{\varepsilon })$. By applying different
results in the literature (see, e.g., \cite{LionsOpt}, \cite%
{Fursikov2000OptimalCO}, \cite{Trolstz}, \cite{Glow-Lions}), it is
well-known that there exists a unique optimal control $v_{\varepsilon }\in
L^{2}(0,T;L^{2}(S_{\varepsilon }^{2}))$.

We point out that the consideration of a non-zero target function $u_{T}\in
L^{2}(\Omega _{\varepsilon })$, a given profile observed at the final time $T
$, can be reduced to the above case of $u_{T}\equiv 0$ by a suitable change
of variables, at least for a dense set of $u_{T}$ in $L^{2}(\Omega )$ (see
Remark \ref{Rem target nonzero} below).


The main goal of this paper is to apply a homogenization process to the
above optimal control problem when $\varepsilon \rightarrow 0$. As in many
other formulations, the kind of the limit problem strongly depends on the
size of the particles' radii $C_{0}\varepsilon ^{\alpha }$, $C_{0}>0$ (see,
e.g. \cite{Timof}, \cite{Anguiano} and the exposition made in \cite%
{DiGoShBook}). Here, we consider the critical case in which $\alpha =\gamma =%
{n}/{(n-2).}$ For different elliptic and parabolic problems, it is
well-known that this critical choice leads to the emergence of a new local
\textquotedblleft strange\textquotedblright \ term (the naming is due to~%
\cite{Cioranescu1997AST}) in the effective partial differential equation
(see~\cite{Hurlov}, \cite{Cioranescu1997AST}, \cite{Kaizu1}, \cite%
{ZuShDiffEq}, and the monograph \cite{DiGoShBook}).

It is well-known that the introduction of a dynamic boundary condition on
the particle boundary causes the aforementioned \textquotedblleft
strange\textquotedblright \ term \ to become a \textquotedblleft
non-local\textquotedblright \ operator, derived by solving a suitable
ordinary differential equation (we refer to \cite{ShDyn19}, \cite%
{DiShZarbitrary}, for the case of an elliptic Poisson equation for the
state). Also, it was shown that in the framework of optimal control problems
there appear some new terms in the limit cost functional (in contrast with
previous results in the literature for related formulations, (see, e.g., ~%
\cite{SJPZ02}, \cite{ZuShDynDok19}, \cite{PoShControl20}, \cite%
{Shaposhnikova2022HomogenizationOT}, \cite{DShPo-first control}, \cite%
{DiPoSh control arbt size}, \cite{DiPoShBound22} and especially \cite%
{DiPoShAttouch} for the case of distributed controls appearing in the state
equation). One of the major new features we will demonstrate in this article
is that when the controls act on the boundary of some particles, some new
terms appear in the cost functional, the non-local terms in time are of a
different nature and some new non-local in time operators must be introduced.

To state the homogenization results, we need to introduce several auxiliary
problems. On the uncontrolled particles, we use non-local operator $%
M(\varphi )$, arising in previous studies (see \cite{DiPoShAttouch}), that
is defined as a solution to 
\begin{equation}
\left \{ 
\begin{array}{lr}
\partial _{t}M(\varphi )+\mathcal{B}_{n}M(\varphi )=\varphi , & t\in (0,T),
\\ 
M(\varphi )(0)=0, & 
\end{array}%
\right. 
\end{equation}%
where $\mathcal{B}_{n}=(n-2)C_{0}^{-1}$ and $\varphi \in L^{2}(0,T)$ is a
given function, and its adjoint operator $M^{\ast },$ 
\begin{equation}
\left \{ 
\begin{array}{lr}
-\partial _{t}M^{\ast }(\varphi )+\mathcal{B}_{n}M^{\ast }(\varphi )=\varphi
, & t\in (0,T), \\ 
M^{\ast }(\varphi )(T)=0. & 
\end{array}%
\right. 
\end{equation}

A similar non-local operator, $G(\varphi )$, and its adjoint operator $%
G^{\ast },$  must be defined on the controlled particles 
\begin{equation}
\left \{ 
\begin{array}{lr}
\partial _{t}G(\varphi )+(\mathcal{B}_{n}+N^{-1})G(\varphi )=\varphi , & 
t\in (0,T), \\ 
G(\varphi )(0)=0,\quad  & 
\end{array}%
\right.   \label{G prob}
\end{equation}
\begin{equation}
\left \{ 
\begin{array}{lr}
-\partial _{t}G^{\ast }(\varphi )+(\mathcal{B}_{n}+N^{-1})G^{\ast }(\varphi
)=\varphi , & t\in (0,T), \\ 
G^{\ast }(\varphi )(T)=0. & 
\end{array}%
\right.   \label{adj G prob}
\end{equation}%
Besides that, we will need to define some new operators $H$ and $H^{\ast },$
coupled with $G^{\ast }$ and $G,$ respectively, by the problems 
\begin{equation}
\left \{ 
\begin{array}{lr}
-\partial _{t}H^{\ast }(\varphi )+(\mathcal{B}_{n}+N^{-1})H^{\ast }(\varphi
)-N^{-1}(\mathcal{B}_{n}+N^{-1})G(H^{\ast }(\varphi ))=\varphi , & t\in
(0,T), \\ 
H^{\ast }(\varphi )(T)=0, & 
\end{array}%
\right.   \label{H*eq}
\end{equation}%
and 
\begin{equation}
\left \{ 
\begin{array}{lr}
\partial _{t}H(\varphi )+(\mathcal{B}_{n}+N^{-1})H(\varphi )-N^{-1}(\mathcal{%
B}_{n}+N^{-1})G^{\ast }(H(\varphi ))=\varphi , & t\in (0,T), \\ 
H(\varphi )(0)=0. & 
\end{array}%
\right.   \label{eqH}
\end{equation}%
%

Notice that the operators $G,$ $M,$ $G^{\ast }$ and $M^{\ast }$ can be
explicitly written. For instance%
\begin{equation*}
G(\varphi )(t)=\int \limits_{0}^{t}e^{-(\mathcal{B}_{n}+N^{-1})(t-s)}\varphi
(s)ds,
\end{equation*}%
which show the non-local in time nature. Some useful properties of these
operators will be shown later (see Section 4).

Although, the detailed statements of our results will be presented later, we
summarize now that\ we will prove the convergence of the optimal controls $%
v_{\varepsilon }$ $\chi _{S_{\varepsilon }^{2,T}}\rightarrow v_{0}\chi
_{\omega ^{T}}$ strongly in $L^{2}(\omega ^{T})$, the convergence of the
corresponding states (extended to $\Omega $) $\tilde{u}_{\varepsilon
}\rightharpoonup u_{0}$ weakly in $L^{2}(0,T;H_{0}^{1}(\Omega ,\partial
\Omega ))$ and $\partial _{t}\tilde{u}_{\varepsilon }\rightharpoonup
\partial _{t}u_{0}$ weakly in $L^{2}(Q^{T})$, and in some sense, that will
be indicated later, the microscopic optimal control $v_{\varepsilon }$
converges to the macroscopic optimal control $v_{0}\in
H^{1}(0,T;L^{2}(\omega )),$ with the limit state problem given by 
\begin{equation}
\left \{ 
\begin{array}{ll}
\partial _{t}u_{0}(v)-\Delta u_{0}(v)+\mathcal{A}_{n}(u_{0}(v)-\mathcal{B}%
_{n}H(u_{0}(v)))\chi _{\omega ^{T}} &  \\ 
+\mathcal{A}_{n}(u_{0}(v)-\mathcal{B}_{n}M(u_{0}(v)))\chi _{(\Omega
\setminus \overline{\omega })\times (0,T)}=f+\mathcal{A}_{n}\mathcal{B}%
_{n}v\chi _{\omega ^{T}}, & (x,t)\in Q^{T}, \\ 
u_{0}(v)(x,0)=0, & x\in \Omega , \\ 
u_{0}(v)(x,t)=0, & (x,t)\in \Gamma ^{T},%
\end{array}%
\right. 
\end{equation}%
where $v\in H^{1}(0,T;L^{2}(\omega ))$ with $v(x,0)=0$ and the limit cost
functional given by 
\begin{eqnarray}
J_{0}(v) &=&\Vert \nabla u_{0}(v)\Vert _{L^{2}(Q^{T})}^{2}+\Vert
u_{0}(v)(x,T)\Vert _{L^{2}(\Omega )}^{2}+N^{-1}\mathcal{A}_{n}\mathcal{B}%
_{n}\int \limits_{\omega ^{T}}(\partial _{t}G^{\ast }(H(u_{0}(v))))^{2}{dxdt}
\notag \\
&&+\mathcal{A}_{n}\mathcal{B}_{n}\int \limits_{\Omega \setminus \overline{%
\omega }}|M(u_{0}(v))(x,T)|^{2}{dx}+\mathcal{A}_{n}\int \limits_{(\Omega
\setminus \overline{\omega })\times (0,T)}|u_{0}(v)-\mathcal{B}%
_{n}M(u_{0}(v))|^{2}{dxdt}  \notag \\
&&+\mathcal{A}_{n}\mathcal{B}_{n}\int \limits_{\omega }|H(u_{0}(v))(x,T)|^{2}{%
dx}+\mathcal{A}_{n}\int \limits_{\omega ^{T}}|u_{0}(v)-\mathcal{B}%
_{n}H(u_{0}(v))|^{2}{dxdt}  \notag \\
&&+N\mathcal{A}_{n}\mathcal{B}_{n}\int \limits_{\omega ^{T}}(\partial
_{t}v)^{2}{dxdt}+\mathcal{A}_{n}\mathcal{B}_{n}(\mathcal{B}%
_{n}+N^{-1})\int \limits_{\omega }v^{2}(x,T){dx}+  \notag \\
&&+N\mathcal{A}_{n}\mathcal{B}_{n}^{2}(\mathcal{B}_{n}+N^{-1})\int \limits_{%
\omega ^{T}}v^{2}{dxdt}.  \label{limit functional}
\end{eqnarray}%
of the optimal control problem 
\begin{equation}
J_{0}(v_{0})=\min \limits_{v\in U_{ad}}J_{0}(v),
\end{equation}%
where the set of admissible functions is now 
\begin{equation*}
U_{ad}=\{ \psi \in H^{1}(0,T;L^{2}(\omega ))|\psi (x,0)=0\}.
\end{equation*}

It can be seen that the first two terms and the last term of $J_{0}$ clearly
correspond to the three terms present in $J_{\varepsilon }$, but the rest of
the terms of $J_{0}$ are, in some way, unexpected. The terms of $J_{0}$
which are related to the final evaluation at time $T$ \ are new, and two of
them are actually non-local in time since they involve the operators $M$ and 
$H$, respectively. The terms of $J_{0}$ which contain the operator $M$ are
integrals extended on the complementary of $\omega $, and they are a
consequence of the microscopic control $v_{\varepsilon }$ being applied only
at the boundary of some particles, $S_{\varepsilon }^{2},$ and not at all of
them. The unexpected terms of $J_{0}$ appear as a consequence of several
implicit relations that are justified in the proof of the Theorem \ref{thm:
cost func limit} below. The last set of the terms that affect time
derivatives of a function of $u_{0}$ and the control $v$ are very surprising
since nothing suggests their appearance when observing the expression for $%
J_{\varepsilon }.$

In order to get the proof of these convergence results, we will use the
extension of the Pontryagin's method to the case of boundary controls (see,
e.g., \cite{LionsOpt}). In Section 2, we give the details of the formulation
of the direct problem as well the coupled system arising in terms of the
adjoint optimal state $p_{\varepsilon }$: we will show that the optimal
control is given by $v_{\varepsilon }=-N^{-1}p_{\varepsilon }\chi
_{S_{\varepsilon }^{2,T}}.$ The a priori estimates allow passing to the
limit in the couple $(u_{\varepsilon },p_{\varepsilon })$ (and thus in the
controls $v_{\varepsilon }$) are obtained in Section 3. Some detailed
statements of the main theorems of this paper are collected in Section 5,
but before that, we present in Section 4 some properties of the auxiliary
non-local in time operators $G$, $H$ and $M$ defined above. The proof
characterizing the limit couple $(u_{0},p_{0})$ from the microscopic couple $%
(u_{\varepsilon },p_{\varepsilon })$ is given in Section 6. Finally, the
identification of the limit cost functional $J_{0}(v)$ from the microscopic
cost functional $J_{\varepsilon }(v)$ is obtained in Section 7.

\section{Problem statement and the adjoint problem}

Let $\Omega $ be a bounded domain in $\mathbb{R}^{n}$, $n\geq 3$, with a
smooth boundary $\partial \Omega $. We denote the unit cube $(-1/2,1/2)^{n}$
centered at the coordinates origin as $Y$. Let $G_{0}$ be a ball of radii $%
C_{0}$ such that $\overline{G}_{0}\subset Y$. Next, given a set $B$ of $%
\mathbb{R}^{n},$ by $\delta B$, $\delta >0$, we denote the set $\{x\in 
\mathbb{R}^{n}|\delta ^{-1}x\in B\}$. For $\varepsilon >0$, we define $%
\widetilde{\Omega }_{\varepsilon }=\{x\in \Omega |\rho (x,\partial \Omega
)>2\varepsilon \}$, where $\rho $ is the Euclidean distance. Let $%
a_{\varepsilon }=C_{0}\varepsilon ^{\alpha }$, where $C_{0}$ is a positive
constant and $\alpha =\frac{n}{n-2}$. We define sets $G_{\varepsilon
}^{j}=a_{\varepsilon }G_{0}+j$, where $j\in \mathbb{Z}^{n}$, $\mathbb{Z}^{n}$
is the set of vectors in $\mathbb{R}^{n}$ with integer coordinates. Now, we
introduce the set of indices $\Upsilon _{\varepsilon }=\{j\in \mathbb{Z}%
^{n}:(a_{\varepsilon }G_{0}+{\varepsilon }j)\bigcap \widetilde{\Omega
_{\varepsilon }}\neq \emptyset \}$, note that the cardinal of $\Upsilon
_{\varepsilon }$ satisfies that $|\Upsilon _{\varepsilon }|\cong d{%
\varepsilon }^{-n}$ for some $d=const>0$. Finally, we define the set 
\begin{equation*}
G_{\varepsilon }=\bigcup_{j\in \Upsilon _{\varepsilon }}G_{\varepsilon }^{j}.
\end{equation*}%
Now, if we define $Y_{\varepsilon }^{j}={\varepsilon }Y+{\varepsilon }j$, $%
P_{\varepsilon }^{j}={\varepsilon }j$, where $Y=(-1/2,1/2)^{n}$, then it is
easy to see that $\overline{G_{\varepsilon }^{j}}\subset Y_{\varepsilon
}^{j} $ and the center of the ball $G_{\varepsilon }^{j}=a_{\varepsilon
}G_{0}+{\varepsilon }j$ coincides with the center of the cube $%
Y_{\varepsilon }^{j}$.

In the formulation of the optimal control problem, we will consider only
some controllable region $\omega$, $\overline{\omega}\subset \Omega$, in the
whole domain $\Omega$. Thus, we split indices of $\Upsilon_\varepsilon$ into
two subsets $\Upsilon^{2}_{\varepsilon}=\{j\in \Upsilon_{\varepsilon}:%
\overline{Y^{j}_{\varepsilon}}\subset \omega \}$ and $\Upsilon^{1}_%
\varepsilon = \Upsilon_\varepsilon \setminus \Upsilon^{2}_\varepsilon$.
Based on these sets, we will use the following notations 
\begin{gather*}
G_{\varepsilon}^{1}=\bigcup_{j\in \Upsilon^{1}_{\varepsilon}}{%
G^{j}_{\varepsilon}},\, G_{\varepsilon}^{2}=\bigcup_{j\in
\Upsilon^{2}_{\varepsilon}}{G^{j}_{\varepsilon}}, \\
S^{1}_\varepsilon = \partial G^{1}_\varepsilon,\,S^{2}_\varepsilon =
\partial G^{2}_\varepsilon.
\end{gather*}

Further, we introduce the sets 
\begin{equation*}
\Omega_{\varepsilon}=\Omega \setminus \overline{G_{\varepsilon}},\quad
\partial \Omega_{\varepsilon}=\partial \Omega \cup S_{\varepsilon},\quad
S_{\varepsilon}=S_{\varepsilon}^{1}\cup S_{\varepsilon}^{2},
\end{equation*}
and, for $0<T<\infty$, we define 
\begin{equation*}
Q^{T}_{\varepsilon}=\Omega_{\varepsilon}\times (0,T),\, \, \omega^{T}=\omega
\times (0,T),
\end{equation*}
\begin{equation*}
S_{\varepsilon}^{T}=S_{\varepsilon}\times
(0,T),\,S_{\varepsilon}^{i,T}=S_{\varepsilon}^{i}\times (0,T),\,i=1,2.
\end{equation*}

Now, we are in a position to formulate optimal control problem. Let $v\in
L^{2}(0,T;S_{\varepsilon }^{2})$. By $u_{\varepsilon }(v)$, we denote an
element of $L^{2}(0,T;H^{1}(\Omega _{\varepsilon },\partial \Omega ))$ with
the time derivative satisfying $\partial _{t}u_{\varepsilon }(v)\in
L^{2}(0,T;L^{2}(\Omega _{\varepsilon }))\bigcap
L^{2}(0,T;L^{2}(S_{\varepsilon }))$ and $u_{\varepsilon }(x,0)=0$ for $x\in
\Omega _{\varepsilon }\bigcup S_{\varepsilon }$, that is a solution to the
parabolic problem with the internal dynamic boundary condition. By $%
H^{1}(\Omega _{\varepsilon },\partial \Omega )$, we denote the closure with
respect to the norm $H^{1}(\Omega _{\varepsilon })$ of the set of infinitely
differentiable in $\overline{\Omega }_{\varepsilon }$ functions vanishing
near the boundary $\partial \Omega $. As a solution of 
\eqref{init state
prob}, we will consider a function $u_{\varepsilon }(v)$ with the
above-mentioned properties that satisfies the following integral identity 
\begin{equation}
\int \limits_{Q_{\varepsilon }^{T}}\partial _{t}u_{\varepsilon }\varphi {dxdt%
}+\int \limits_{Q_{\varepsilon }^{T}}\nabla u_{\varepsilon }\nabla \varphi {%
dxdt}+\varepsilon ^{-\gamma }\int \limits_{S_{\varepsilon }^{T}}\partial
_{t}u_{\varepsilon }\varphi {dsdt}=\int \limits_{Q_{\varepsilon
}^{T}}f\varphi{dxdt}+\varepsilon ^{-\gamma }\int \limits_{S_{\varepsilon
}^{2,T}}v\varphi {dsdt}  \label{int ident u}
\end{equation}%
for an arbitrary function $\varphi \in L^{2}(0,T;H^{1}(\Omega _{\varepsilon
},\partial \Omega ))$. We consider now the optimal control problem stated in
the Introduction (see \eqref{optim control def}).

\begin{remark}
\label{Rem target nonzero}Our approach can be easily extended to the case of
a non-zero target $u_{T}\in L^{2}(\Omega )$, at least for a dense set of $%
u_{T}$ in $L^{2}(\Omega )$, i.e. the cost functional will be%
\begin{eqnarray}
J_{\varepsilon }(v) &=&\frac{1}{2}\Vert \nabla {u_{\varepsilon }(v)}\Vert
_{L^{2}(Q_{\varepsilon }^{T})}^{2}+\frac{1}{2}\int \limits_{\Omega
_{\varepsilon }}\left( u_{\varepsilon }(v)(x,T)-u_{T}\right) ^{2}dx  \notag
\\
&&+\frac{{\varepsilon }^{-\gamma }}{2}\int \limits_{S_{\varepsilon
}}u_{\varepsilon }^{2}(v)(x,T){ds}+\varepsilon ^{-\gamma }\frac{N}{2}\Vert {v%
}\Vert _{L^{2}(S_{\varepsilon }^{2,T})}^{2}.  \label{nonzero cost}
\end{eqnarray}

Indeed, let us assume that $u_{T}\in L^{2}(\Omega )$ is such that there
exists a converging as $\varepsilon \to0$ sequence of functions $%
V_{\varepsilon }\in L^{2}(0,T;L^{2}(\Omega _{\varepsilon }))$, i.e. 
\begin{equation}
V_{\varepsilon }\rightharpoonup V_{0}\mbox{ weakly in }L^{2}(Q^{T}),\text{
for some }V_{0}\in L^{2}(Q^{T}),  \label{hypo Weak  conver Ve}
\end{equation}
such that for the unique solution $U_{\varepsilon }$ of the auxiliary
problem 
\begin{equation}
\left \{ 
\begin{array}{lr}
\partial _{t}U_{\varepsilon }-\Delta U_{\varepsilon }=V_{\varepsilon }, & 
(x,t)\in Q_{\varepsilon }^{T}, \\ 
\varepsilon ^{-\gamma }\partial _{t}U_{\varepsilon }+\partial _{\nu
}U_{\varepsilon }=0, & (x,t)\in S_{\varepsilon }^{T}, \\ 
U_{\varepsilon }(x,0)=0, & x\in \Omega _{\varepsilon }\bigcup S_{\varepsilon
}, \\ 
U_{\varepsilon }(x,t)=0, & (x,t)\in \Gamma ^{T},%
\end{array}%
\right.
\end{equation}%
we have%
\begin{equation}
\widetilde{U}_{\varepsilon }\rightharpoonup U_{0}\mbox{ weakly in }%
L^{2}(0,T;H_{0}^{1}(\Omega )), \\
\partial _{t}\widetilde{U}_{\varepsilon }\rightharpoonup \partial _{t}U_{0}%
\mbox{ weakly in }L^{2}(Q^{T}),
\end{equation}%
and 
\begin{equation*}
U_{0}(x,T)=u_{T}(x)\, \text{ a.e. }\, x\in \Omega .
\end{equation*}

Then by defining the change of variables 
\begin{equation*}
w{_{\varepsilon }(v)=u_{\varepsilon }(v)-}U_{\varepsilon },
\end{equation*}%
where now ${u_{\varepsilon }(v)}$ is the optimal control associated to the
cost functional (\ref{nonzero cost}), we find that $w{_{\varepsilon }(v)}$
is the optimal control associate to the previous cost functional (\ref{cost
functional}) with $u_{T}\equiv 0.$ Finally, by the arguments of Remark 7.1
of \cite{DiPoShAttouch} (or Theorem 4 of \  \cite{DiPoShRACSAM}), it is easy
to prove that the set of final data $u_{T}\in L^{2}(\Omega )$ satisfying the
above mentioned conditions is a dense set of\ $L^{2}(\Omega ).$ Then, the
perturbed equation satisfied by $w{_{\varepsilon }(v)}$, i.e. 
\begin{equation*}
\partial _{t}w{_{\varepsilon }(v)}-\Delta w{_{\varepsilon }(v)}%
=f-V_{\varepsilon },
\end{equation*}%
does not add any difficulty, once we know that (\ref{hypo Weak conver Ve})
holds.
\end{remark}

\bigskip

To obtain the characterization of the optimal control, we consider the
adjoint problem 
\begin{equation}
\left \{ 
\begin{array}{lr}
-\partial _{t}p_{\varepsilon }-\Delta p_{\varepsilon }=-\Delta
u_{\varepsilon }, & (x,t)\in Q_{\varepsilon }^{T}, \\ 
\partial _{\nu }p_{\varepsilon }-\varepsilon ^{-\gamma }\partial
_{t}p_{\varepsilon }=\partial _{\nu }u_{\varepsilon }, & (x,t)\in
S_{\varepsilon }^{T}, \\ 
p_{\varepsilon }(x,T)=u_{\varepsilon }(x,T), & x\in \Omega _{\varepsilon
}\bigcup S_{\varepsilon }, \\ 
p_{\varepsilon }(x,t)=0, & (x,t)\in \Gamma ^{T}.%
\end{array}%
\right.  \label{init adjoint prob}
\end{equation}

We say that a function $p_{\varepsilon }\in L^{2}(0,T;H^{1}(\Omega
_{\varepsilon },\partial \Omega )),$ with $\partial _{t}p_{\varepsilon }\in
L^{2}(0,T;L^{2}(\Omega _{\varepsilon }))\cap L^{2}(0,T;L^{2}(S_{\varepsilon
})),$ is a weak solution to the problem \eqref{init adjoint prob} if $%
p_{\varepsilon }(x,T)=u_{\varepsilon }(x,T)$ for a.e. $x\in \Omega
_{\varepsilon }$ and a.e. $x\in S_{\varepsilon }$ and if it satisfies the
integral identity 
\begin{equation}
-\int \limits_{Q_{\varepsilon }^{T}}\partial _{t}p_{\varepsilon }\varphi {%
dxdt}+\int \limits_{Q_{\varepsilon }^{T}}\nabla p_{\varepsilon }\nabla
\varphi {dxdt}-\varepsilon ^{-\gamma }\int \limits_{S_{\varepsilon
}^{T}}\partial _{t}p_{\varepsilon }\varphi {dsdt}=\int
\limits_{Q_{\varepsilon }^{T}}\nabla u_{\varepsilon }\nabla \varphi {dxdt},
\label{int ident init adjoint prob}
\end{equation}%
for any test function $\varphi \in L^{2}(0,T;H^{1}(\Omega _{\varepsilon
},\partial \Omega ))$. For a given $u_{\varepsilon }$ (with the regularity
of the weak solutions of (\ref{init state prob})) it is well-known that
there exists a unique solution to the problem \eqref{init adjoint prob}
(see, e.g., \cite{BejeDVrabie} and its references).

The following theorem gives a characterization of the optimal control $%
v_{\varepsilon }$ in terms of the adjoint state $p_{\varepsilon }$.

\begin{theorem}
\label{thm: init opt cont} Let the pair of functions $(u_\varepsilon(v_%
\varepsilon), v_\varepsilon)$ be an optimal solution of the problem %
\eqref{optim control def}, then $v_\varepsilon = -N^{-1}p_\varepsilon
\chi_{S^{2, T}_\varepsilon}$, where $p_\varepsilon$ is the solution to %
\eqref{init adjoint prob}. The converse is also true.
\end{theorem}

\begin{proof}
Let $v$ be an arbitrary function in $L^{2}(S^{2 ,T} _{\varepsilon })$ and $%
\lambda >0$. By $u_{\varepsilon }^{\lambda }$, we denote the solution of %
\eqref{optim control def} with the control $v_{\varepsilon }+\lambda v$,
i.e. $u_{\varepsilon }^{\lambda }=u_{\varepsilon }(v_{\varepsilon }+\lambda
v)$. We use $u_{\varepsilon }=u_{\varepsilon }(v_{\varepsilon })$ to
simplify the notation. Then we have 
\begin{gather*}
J_{\varepsilon }(v_{\varepsilon }+\lambda v)-J_{\varepsilon }(v_{\varepsilon
})=\frac{1}{2}\Vert \nabla u_{\varepsilon }^{\lambda }\Vert
_{L^{2}(Q_{\varepsilon }^{T})}^{2}+\frac{1}{2}\Vert u_{\varepsilon
}^{\lambda }(x,T)\Vert _{L^{2}(\Omega _{\varepsilon })}^{2} \\
+\frac{\varepsilon ^{-\gamma }}{2}\Vert u_{\varepsilon }^{\lambda
}(x,T)\Vert _{L^{2}(S_{\varepsilon })}^{2}+\varepsilon ^{-\gamma }\frac{N}{2}%
\Vert v_{\varepsilon }+\lambda v\Vert _{L^{2}(S_{\varepsilon }^{2,T})}^{2} \\
-\frac{1}{2}\Vert \nabla u_{\varepsilon }\Vert _{L^{2}(Q_{\varepsilon
}^{T})}^{2}-\frac{1}{2}\Vert u_{\varepsilon }(x,T)\Vert _{L^{2}(\Omega
_{\varepsilon })}^{2} \\
-\frac{\varepsilon ^{-\gamma }}{2}\Vert u_{\varepsilon }(x,T)\Vert
_{L^{2}(S_{\varepsilon })}^{2}-\varepsilon ^{-\gamma }\frac{N}{2}\Vert
v_{\varepsilon }\Vert _{L^{2}(S_{\varepsilon }^{2,T})}^{2} \\
=\frac{1}{2}\int \limits_{Q_{\varepsilon }^{T}}\nabla (u_{\varepsilon
}^{\lambda }-u_{\varepsilon })\nabla (u_{\varepsilon }^{\lambda
}+u_{\varepsilon }){dxdt}+\frac{1}{2}\int \limits_{\Omega _{\varepsilon
}}(u_{\varepsilon }^{\lambda }-u_{\varepsilon })(x,T)(u_{\varepsilon
}^{\lambda }+u_{\varepsilon })(x,T){dx} \\
+\frac{\varepsilon ^{-\gamma }}{2}\int \limits_{S_{\varepsilon
}}(u_{\varepsilon }^{\lambda }-u_{\varepsilon })(x,T)(u_{\varepsilon
}^{\lambda }+u_{\varepsilon })(x,T){ds}+\varepsilon ^{-\gamma }\frac{N}{2}%
\int \limits_{S _{\varepsilon }^{2, T}}(2\lambda v_{\varepsilon }v+\lambda
^{2}v^{2}){dsdt.}
\end{gather*}%
We define the function $\theta _{\varepsilon }=(u_{\varepsilon }^{\lambda
}-u_{\varepsilon })/\lambda $. It is easy to see that $\theta _{\varepsilon
} $ is the unique solution to the problem 
\begin{equation*}
\left \{ 
\begin{array}{lr}
\partial _{t}\theta _{\varepsilon }-\Delta \theta _{\varepsilon }=0, & 
(x,t)\in Q_{\varepsilon }^{T}, \\ 
\varepsilon ^{-\gamma }\partial _{t}\theta _{\varepsilon }+\partial _{\nu
}\theta _{\varepsilon }=\chi _{S_{\varepsilon }^{2,T}}\varepsilon ^{-\gamma
}v, & (x,t)\in S_{\varepsilon }^{T}, \\ 
\theta _{\varepsilon }(x,0)=0, & x\in \Omega _{\varepsilon }\bigcup
S_{\varepsilon }, \\ 
\theta _{\varepsilon }(x,t)=0, & (x,t)\in \Gamma ^{T}.%
\end{array}%
\right.
\end{equation*}%
Using the definition of $\theta _{\varepsilon }$, we have 
\begin{gather*}
J_{\varepsilon }^{\prime }(v_{\varepsilon })v=\lim \limits_{\lambda
\rightarrow 0}(J_{\varepsilon }(v_{\varepsilon }+\lambda v)-J_{\varepsilon
}(v_{\varepsilon }))/\lambda = \\
=\int \limits_{Q_{\varepsilon }^{T}}\nabla \theta _{\varepsilon }\nabla
u_{\varepsilon }{dxdt}+\int \limits_{\Omega _{\varepsilon }}\theta
_{\varepsilon }(x,T)u_{\varepsilon }(x,T){dx}+ \\
+\varepsilon ^{-\gamma }\int \limits_{S_{\varepsilon }}\theta _{\varepsilon
}(x,T)u_{\varepsilon }(x,T){ds}+\varepsilon ^{-\gamma }N\int
\limits_{S_{\varepsilon }^{2,T}}v_{\varepsilon }v{dsdt}.
\end{gather*}%
Now, we use the definition of $p_{\varepsilon }$ and derive from the last
expression the identity 
\begin{equation*}
J_{\varepsilon }^{\prime }(v_{\varepsilon })v=\varepsilon ^{-\gamma }\int
\limits_{S_{\varepsilon }^{2,T}}p_{\varepsilon }v{dxdt}+\varepsilon
^{-\gamma }N\int \limits_{S_{\varepsilon }^{2,T}}v_{\varepsilon }v{dsdt}.
\end{equation*}%
As $v_{\varepsilon }$ is the optimal control, we should have $J_{\varepsilon
}^{\prime }(v_{\varepsilon })\cdot v=0$ for all $v\in
L^{2}(0,T;L^{2}(S_{\varepsilon }^{2}))$. Hence, $v_{\varepsilon
}=-N^{-1}p_{\varepsilon }$ for a.e. $(x,t)\in S_{\varepsilon }^{2,T}$. This
concludes the proof.
\end{proof}

In consequence, by Theorem~\ref{thm: init opt cont}, the optimal control
problem is characterized through the coupled system 
\begin{equation}
\left \{ 
\begin{array}{lr}
\partial _{t}u_{\varepsilon }-\Delta u_{\varepsilon }=f, & (x,t)\in
Q_{\varepsilon }^{T}, \\ 
-\partial _{t}p_{\varepsilon }-\Delta p_{\varepsilon }=-\Delta
u_{\varepsilon }, & (x,t)\in Q_{\varepsilon }^{T}, \\ 
\partial _{\nu }u_{\varepsilon }+\varepsilon ^{-\gamma }\partial
_{t}u_{\varepsilon }=-\varepsilon ^{-\gamma }N^{-1}\chi _{S_{\varepsilon
}^{2, T}}p_{\varepsilon }, & (x,t)\in S_{\varepsilon }^{T}, \\ 
\partial _{\nu }p_{\varepsilon }-\varepsilon ^{-\gamma }\partial
_{t}p_{\varepsilon }=\partial _{\nu }u_{\varepsilon }, & (x,t)\in
S_{\varepsilon }^{T}, \\ 
u_{\varepsilon }(x,0)=0, & x\in \Omega _{\varepsilon }\bigcup S_{\varepsilon
}, \\ 
p_{\varepsilon }(x,T)=u_{\varepsilon }(x,T), & x\in \Omega _{\varepsilon
}\bigcup S_{\varepsilon }, \\ 
u_{\varepsilon }(x,t)=p_{\varepsilon }(x,t)=0, & (x,t)\in \Gamma ^{T}.%
\end{array}%
\right.  \label{init coupled prob}
\end{equation}

\section{A priori estimates}

In this section, we get several a priori estimates of the state and adjoint
state. Taking $p_{\varepsilon }$ as a test function in the integral identity
for $u_{\varepsilon }$, we get 
\begin{equation}
\int \limits_{Q_{\varepsilon }^{T}}\partial _{t}u_{\varepsilon
}p_{\varepsilon }{dxdt}+\varepsilon ^{-\gamma }\int \limits_{S_{\varepsilon
}^{T}}\partial _{t}u_{\varepsilon }p_{\varepsilon }{dsdt}+\int \limits_{Q_{%
\varepsilon }^{T}}\nabla u_{\varepsilon }\nabla p_{\varepsilon }{dxdt} \\
=\int \limits_{Q_{\varepsilon }^{T}}fp_{\varepsilon }{dxdt}-N^{-1}\varepsilon
^{-\gamma }\int \limits_{S_{\varepsilon }^{2,T}}p_{\varepsilon }^{2}{dxdt}.
\label{apr estim: int ident u p}
\end{equation}%
Now, taking $u_{\varepsilon }$ as a test function in the integral identity
for $p_{\varepsilon }$, we get 
\begin{equation}
-\int \limits_{Q_{\varepsilon }^{T}}\partial _{t}p_{\varepsilon
}u_{\varepsilon }{dxdt}-\varepsilon ^{-\gamma }\int \limits_{S_{\varepsilon
}^{T}}\partial _{t}p_{\varepsilon }u_{\varepsilon }{dsdt}+\int \limits_{Q_{%
\varepsilon }^{T}}\nabla p_{\varepsilon }\nabla u_{\varepsilon }{dxdt}%
=\int \limits_{Q_{\varepsilon }^{T}}|\nabla u_{\varepsilon }|^{2}{dxdt}.
\label{apr estim: int ident p u}
\end{equation}%
Next, we subtract \eqref{apr estim: int ident u p} from 
\eqref{apr estim:
int ident p u} and obtain the expression 
\begin{equation*}
-\int \limits_{Q_{\varepsilon }^{T}}\partial _{t}(u_{\varepsilon
}p_{\varepsilon }){dxdt}-\varepsilon ^{-\gamma }\int \limits_{S_{\varepsilon
}^{T}}\partial _{t}(u_{\varepsilon }p_{\varepsilon }){dsdt}%
-N^{-1}\varepsilon ^{-\gamma }\int \limits_{S_{\varepsilon
}^{2,T}}p_{\varepsilon }^{2}{dxdt} \\
=\int \limits_{Q_{\varepsilon }^{T}}|\nabla u_{\varepsilon }|^{2}{dxdt}%
-\int \limits_{Q_{\varepsilon }^{T}}fp_{\varepsilon }{dxdt}.
\end{equation*}%
From here, we get 
\begin{gather}
\Vert \nabla u_{\varepsilon }\Vert _{L^{2}(Q_{\varepsilon }^{T})}^{2}+\Vert
u_{\varepsilon }(x,T)\Vert _{L^{2}(\Omega _{\varepsilon })}^{2}+\varepsilon
^{-\gamma }\Vert u_{\varepsilon }(x,T)\Vert _{L^{2}(S_{\varepsilon })}^{2} 
\notag \\
+N^{-1}\varepsilon ^{-\gamma }\int \limits_{S_{\varepsilon
}^{2,T}}p_{\varepsilon }^{2}{dsdt}\leq \int \limits_{Q_{\varepsilon
}^{T}}|f||p_{\varepsilon }|{dxdt}.  \label{apr estim: u estim 1}
\end{gather}%
%
%
%
%
Then, we take $p_{\varepsilon }$ as a test function in the integral identity %
\eqref{int ident init adjoint prob}, and get 
\begin{equation*}
-\frac{1}{2}\Vert u_{\varepsilon }(x,T)\Vert _{L^{2}(\Omega _{\varepsilon
})}^{2}-\frac{\varepsilon ^{-\gamma }}{2}\Vert u_{\varepsilon }(x,T)\Vert
_{L^{2}(S_{\varepsilon })}^{2}+\Vert \nabla p_{\varepsilon }\Vert
_{L^{2}(Q_{\varepsilon }^{T})}^{2}\leq \int \limits_{Q_{\varepsilon
}^{T}}\nabla u_{\varepsilon }\nabla p_{\varepsilon }{dxdt}.
\end{equation*}%
From here and \eqref{apr estim: u estim 1}, we conclude 
\begin{gather}
\Vert \nabla p_{\varepsilon }\Vert _{L^{2}(Q_{\varepsilon }^{T})}^{2}\leq
C(\Vert \nabla u_{\varepsilon }\Vert _{L^{2}(Q_{\varepsilon
}^{T})}^{2}+\Vert u_{\varepsilon }(x,T)\Vert _{L^{2}(\Omega _{\varepsilon
})}^{2}+\varepsilon ^{-\gamma }\Vert u_{\varepsilon }(x,T)\Vert
_{L^{2}(S_{\varepsilon })}^{2})  \notag \\
\leq C\int \limits_{Q_{\varepsilon }^{T}}|f||p_{\varepsilon }|{dxdt}.
\label{prelim p estim}
\end{gather}%
Here and below, constant $C$ is independent from $\varepsilon $. As $%
p_{\varepsilon }$ is in $H^{1}(\Omega _{\varepsilon },\partial \Omega )$, we
can apply Poincar\'{e}-Friedrichs's inequality 
\begin{equation*}
\Vert p_{\varepsilon }(.,t)\Vert _{L^{2}(\Omega _{\varepsilon })}\leq K\Vert
\nabla p_{\varepsilon }(.,t)\Vert _{L^{2}(\Omega _{\varepsilon })}.
\end{equation*}%
Using this inequality in the previous estimate \eqref{prelim p estim}
, we get 
\begin{equation*}
\Vert p_{\varepsilon }\Vert _{L^{2}(Q_{\varepsilon }^{T})}^{2}\leq C\Vert
f\Vert _{L^{2}(Q^{T})}^{2}.
\end{equation*}%
Now, we substitute this estimate into \eqref{apr estim: u estim 1}, and
derive the following estimate of $u_{\varepsilon }$ 
\begin{equation*}
\Vert \nabla u_{\varepsilon }\Vert _{L^{2}(Q_{\varepsilon }^{T})}^{2}+\Vert
u_{\varepsilon }(x,T)\Vert _{L^{2}(\Omega _{\varepsilon })}^{2}+\varepsilon
^{-\gamma }\Vert u_{\varepsilon }(x,T)\Vert _{L^{2}(S_{\varepsilon
})}^{2}+N^{-1}\varepsilon ^{-\gamma }\Vert p_{\varepsilon }\Vert
_{L^{2}(S_{\varepsilon }^{2,T})}^{2}\leq C\Vert f\Vert _{L^{2}(Q^{T})}^{2}.
\end{equation*}%
From here, by \eqref{prelim p estim}, we get the estimation of the gradient
of $p_{\varepsilon }$ 
\begin{equation*}
\Vert \nabla p_{\varepsilon }\Vert _{L^{2}(Q_{\varepsilon }^{T})}^{2}\leq
C\Vert f\Vert _{L^{2}(Q^{T})}^{2}.
\end{equation*}%
Now we derive some estimates on the time derivatives of $u_{\varepsilon }$
and $p_{\varepsilon }$. We use Galerkin's approach and construct $%
u_{\varepsilon }^{m}$ and $p_{\varepsilon }^{m}$, where $m=1,2,\dots $, that
are approximations to $u_{\varepsilon }$ and $p_{\varepsilon }$. Note that,
for such approximations, we have the same estimates derived above on $%
u_{\varepsilon }$ and $p_{\varepsilon }$. We take now $\partial
_{t}u_{\varepsilon }^{m}$ as a test function in the equations for $%
u_{\varepsilon }^{m}$, and integrating from $0$ to an arbitrary $\tau \in
\lbrack 0,T]$, we get 
\begin{gather*}
\Vert \partial _{t}u_{\varepsilon }^{m}\Vert _{L^{2}(Q_{\varepsilon
}^{T})}^{2}+\varepsilon ^{-\gamma }\Vert \partial _{t}u_{\varepsilon
}^{m}\Vert _{L^{2}(S_{\varepsilon }^{T})}^{2}+\max \limits_{t\in \lbrack
0,T]}\Vert \nabla u_{\varepsilon }^{m}\Vert _{L^{2}(\Omega _{\varepsilon
})}^{2} \\
\leq K(\int \limits_{Q_{\varepsilon }^{T}}|f||\partial _{t}u_{\varepsilon
}^{m}|{dxdt}+\varepsilon ^{-\gamma }\int \limits_{S_{\varepsilon
}^{2,T}}|p_{\varepsilon }^{m}||\partial _{t}u_{\varepsilon }^{m}|{dsdt}) \\
\leq \frac{1}{2}\Vert \partial _{t}u_{\varepsilon }^{m}\Vert
_{L^{2}(Q_{\varepsilon }^{T})}^{2}+\frac{\varepsilon ^{-\gamma }}{2}\Vert
\partial _{t}u_{\varepsilon }^{m}\Vert _{L^{2}(S_{\varepsilon
}^{T})}^{2}+K(\varepsilon ^{-\gamma }\Vert p_{\varepsilon }^{m}\Vert
_{L^{2}(S_{\varepsilon }^{2,T})}^{2}+\Vert f\Vert _{L^{2}(Q^{T})}^{2}),
\end{gather*}%
where constant $K$ is independent of $\varepsilon $ and $m$. From here, we
immediately derive 
\begin{equation*}
\Vert \partial _{t}u_{\varepsilon }^{m}\Vert _{L^{2}(Q_{\varepsilon
}^{T})}^{2}+\varepsilon ^{-\gamma }\Vert \partial _{t}u_{\varepsilon
}^{m}\Vert _{L^{2}(S_{\varepsilon }^{T})}^{2}+\max \limits_{t\in \lbrack
0,T]}\Vert \nabla u_{\varepsilon }^{m}\Vert _{L^{2}(\Omega _{\varepsilon
})}^{2}\leq K\Vert f\Vert _{L^{2}(Q^{T})}^{2}.
\end{equation*}%
Then, passing to the limit, as $m\rightarrow \infty ,$ in this estimate we
have 
\begin{equation*}
\Vert \partial _{t}u_{\varepsilon }\Vert _{L^{2}(Q_{\varepsilon
}^{T})}^{2}+\varepsilon ^{-\gamma }\Vert \partial _{t}u_{\varepsilon }\Vert
_{L^{2}(S_{\varepsilon }^{T})}^{2}+\max \limits_{t\in \lbrack 0,T]}\Vert
\nabla u_{\varepsilon }\Vert _{L^{2}(\Omega _{\varepsilon })}^{2}\leq K\Vert
f\Vert _{L^{2}(Q^{T})}^{2}.
\end{equation*}%
Moreover, if we use $\partial _{t}p_{\varepsilon }^{m}$ as a test function
in the equation for $p_{\varepsilon }^{m}$, we get, for a.e. $t$ 
\begin{gather*}
-\Vert \partial _{t}p_{\varepsilon }^{m}\Vert _{L^{2}(\Omega _{\varepsilon
})}^{2}-\varepsilon ^{-\gamma }\Vert \partial _{t}p_{\varepsilon }^{m}\Vert
_{L^{2}(S_{\varepsilon })}^{2}+(\nabla p_{\varepsilon }^{m},\partial
_{t}\nabla p_{\varepsilon }^{m})_{L^{2}(\Omega _{\varepsilon })} \\
=-(\partial _{t}u_{\varepsilon }^{m},\partial _{t}p_{\varepsilon
}^{m})_{L^{2}(\Omega _{\varepsilon })}-\varepsilon ^{-\gamma }(\partial
_{t}u_{\varepsilon }^{m},\partial _{t}p_{\varepsilon
}^{m})_{L^{2}(S_{\varepsilon })} \\
+(f,\partial _{t}p_{\varepsilon }^{m})_{L^{2}(\Omega _{\varepsilon
})}-N^{-1}\varepsilon ^{-\gamma }(p_{\varepsilon }^{m},\partial
_{t}p_{\varepsilon }^{m})_{L^{2}(S_{\varepsilon }^{2})}.
\end{gather*}%
Integrating this equality with respect to $t$ from $0$ to $T$, we obtain 
\begin{gather*}
-\Vert \partial _{t}p_{\varepsilon }^{m}\Vert _{L^{2}(Q_{\varepsilon
}^{T})}^{2}+\frac{1}{2}\Vert \nabla p_{\varepsilon }^{m}(\cdot ,T)\Vert
_{L^{2}(\Omega _{\varepsilon })}^{2}-\frac{1}{2}\Vert \nabla p_{\varepsilon
}^{m}(\cdot ,0)\Vert _{L^{2}(\Omega _{\varepsilon })}^{2}-\varepsilon
^{-\gamma }\Vert \partial _{t}p_{\varepsilon }^{m}\Vert
_{L^{2}(S_{\varepsilon }^{T})}^{2} \\
=-\int \limits_{Q_{\varepsilon }^{T}}\partial _{t}u_{\varepsilon
}^{m}\partial _{t}p_{\varepsilon }^{m}{dxdt}-\varepsilon ^{-\gamma
}\int \limits_{S_{\varepsilon }^{T}}\partial _{t}u_{\varepsilon }^{m}\partial
_{t}p_{\varepsilon }^{m}{dsdt} \\
+\int \limits_{Q_{\varepsilon }^{T}}f\partial _{t}p_{\varepsilon }^{m}{dxdt}%
-N^{-1}\varepsilon ^{-\gamma }\int \limits_{S_{\varepsilon
}^{2,T}}p_{\varepsilon }^{m}\partial _{t}p_{\varepsilon }^{m}{dsdt} \\
=-\int \limits_{Q_{\varepsilon }^{T}}\partial _{t}u_{\varepsilon
}^{m}\partial _{t}p_{\varepsilon }^{m}{dxdt}-\varepsilon ^{-\gamma
}\int \limits_{S_{\varepsilon }^{T}}\partial _{t}u_{\varepsilon }^{m}\partial
_{t}p_{\varepsilon }^{m}{dsdt} \\
+\int \limits_{Q_{\varepsilon }^{T}}f\partial _{t}p_{\varepsilon }^{m}{dxdt}%
+N^{-1}\frac{\varepsilon ^{-\gamma }}{2}\Vert p_{\varepsilon }^{m}(x,0)\Vert
_{L^{2}(S_{\varepsilon }^{2})}^{2}-N^{-1}\frac{\varepsilon ^{-\gamma }}{2}%
\Vert u_{\varepsilon }^{m}(x,T)\Vert _{L^{2}(S_{\varepsilon }^{2})}^{2}.
\end{gather*}%
From here, we derive 
\begin{gather*}
\Vert \partial _{t}p_{\varepsilon }^{m}\Vert _{L^{2}(Q_{\varepsilon
}^{T})}^{2}+\varepsilon ^{-\gamma }\Vert \partial _{t}p_{\varepsilon
}^{m}\Vert _{L^{2}(S_{\varepsilon }^{T})}^{2}+\frac{1}{2}\Vert \nabla
p_{\varepsilon }^{m}(\cdot ,0)\Vert _{L^{2}(\Omega _{\varepsilon })}^{2}+%
\frac{\varepsilon ^{-\gamma }}{2N}\Vert p_{\varepsilon }^{m}(x,0)\Vert
_{L^{2}(\omega _{\varepsilon })}^{2} \\
\leq \frac{1}{2}\Vert \nabla u_{\varepsilon }^{m}(\cdot ,T)\Vert
_{L^{2}(\Omega _{\varepsilon })}^{2}+\Vert \partial _{t}u_{\varepsilon
}^{m}\Vert _{L^{2}(Q_{\varepsilon }^{T})}\Vert \partial _{t}p_{\varepsilon
}^{m}\Vert _{L^{2}(Q_{\varepsilon }^{T})} \\
+\varepsilon ^{-\gamma }\Vert \partial _{t}u_{\varepsilon }^{m}\Vert
_{L^{2}(S_{\varepsilon }^{T})}\Vert \partial _{t}p_{\varepsilon }^{m}\Vert
_{L^{2}(S_{\varepsilon }^{T})}+\Vert f\Vert _{L^{2}(Q_{\varepsilon
}^{T})}\Vert \partial _{t}p_{\varepsilon }^{m}\Vert _{L^{2}(Q_{\varepsilon
}^{T})}.
\end{gather*}%
Finally, using estimates obtained for $u_{\varepsilon }^{m}$, we conclude 
\begin{equation}
\Vert \partial _{t}p_{\varepsilon }^{m}\Vert _{L^{2}(Q_{\varepsilon
}^{T})}^{2}+\varepsilon ^{-\gamma }\Vert \partial _{t}p_{\varepsilon
}^{m}\Vert _{L^{2}(S_{\varepsilon }^{T})}^{2}\leq K\Vert f\Vert
_{L^{2}(Q^{T})}^{2}.
\end{equation}%
Passing to the limit, as $m\rightarrow \infty $, we get the estimation of $%
\partial _{t}p_{\varepsilon }$ 
\begin{equation}
\Vert \partial _{t}p_{\varepsilon }\Vert _{L^{2}(Q_{\varepsilon
}^{T})}^{2}+\varepsilon ^{-\gamma }\Vert \partial _{t}p_{\varepsilon }\Vert
_{L^{2}(S_{\varepsilon }^{T})}^{2}\leq K\Vert f\Vert _{L^{2}(Q^{T})}^{2}.
\end{equation}%
Having proved some a priori estimates of $u_{\varepsilon }$ and $%
v_{\varepsilon }$, we proceed with the extension of these solution to the
whole cylinder $Q^{T}$. We know (see, e.g. \cite{DiGoShBook} and its
references) that there exists an extension operator $P_{\varepsilon
}:H^{1}(Q_{\varepsilon }^{T})\rightarrow H^{1}(Q^{T})$ such that the
following estimate holds 
\begin{equation*}
\Vert P_{\varepsilon }(u)\Vert _{H^{1}(Q^{T})}\leq \Vert u\Vert
_{H^{1}(Q_{\varepsilon }^{T})}.
\end{equation*}%
Let $\tilde{u}_{\varepsilon }$, $\tilde{p}_{\varepsilon }$ be the extensions
of the functions $u_{\varepsilon }$, $p_{\varepsilon }$. Then we get the
following estimates 
\begin{gather}
\Vert \partial _{t}\tilde{u}_{\varepsilon }\Vert _{L^{2}(Q^{T})}^{2}+\Vert
\nabla \tilde{u}_{\varepsilon }\Vert _{L^{2}(Q^{T})}^{2}\leq K(\Vert
\partial _{t}u_{\varepsilon }\Vert _{L^{2}(Q_{\varepsilon }^{T})}^{2}+\Vert
\nabla u_{\varepsilon }\Vert _{L^{2}(Q_{\varepsilon }^{T})}^{2}),
\label{ext u estim} \\
\Vert \partial _{t}\tilde{p}_{\varepsilon }\Vert _{L^{2}(Q^{T})}^{2}+\Vert
\nabla \tilde{p}_{\varepsilon }\Vert _{L^{2}(Q^{T})}^{2}\leq K(\Vert
\partial _{t}p_{\varepsilon }\Vert _{L^{2}(Q_{\varepsilon }^{T})}^{2}+\Vert
\nabla p_{\varepsilon }\Vert _{L^{2}(Q_{\varepsilon }^{T})}^{2}).
\label{ext p estim}
\end{gather}%
The obtained estimates imply that there exist some subsequences (still
denoted as the original) and some limit functions, $u_{0}$ and $p_{0},$ such
that 
\begin{equation}
\begin{gathered} \tilde{u}_{\varepsilon }\rightharpoonup u_{0}\mbox{ weakly in }L^{2}(0,T;H_{0}^{1}(\Omega )),\quad \partial _{t}\tilde{u}_{\varepsilon }\rightharpoonup \partial _{t}u_{0}\mbox{ weakly
in }L^{2}(Q^{T}), \\ \tilde{p}_{\varepsilon }\rightharpoonup p_{0}\mbox{ weakly in }L^{2}(0,T;H_{0}^{1}(\Omega )),\quad \partial _{t}\tilde{p}_{\varepsilon }\rightharpoonup \partial _{t}p_{0}\mbox{ weakly
in }L^{2}(Q^{T}). \end{gathered}  \label{limit func def}
\end{equation}%
Moreover, the embedding theorem implies also that $\tilde{u}_{\varepsilon
}\rightarrow u_{0}$ and $\tilde{p}_{\varepsilon }\rightarrow p_{0}$ in $%
L^{2}(Q^{T})$.

\bigskip

In the rest of the paper we will give the characterization of these limit
functions and derive a formulation of the homogenized optimal control
problem.

\section{Auxiliary non-local in time operators $G,$ $H$ and $M$}

As already been noted in the Introduction, to state the homogenization
results we need to introduce some auxiliary problems. The first non-local
operator $M(\varphi )$ already was used in our previous study related to the
case of distributed controls (\cite{DiPoShAttouch}), 
\begin{equation}
\left \{ 
\begin{array}{lr}
\partial _{t}M(\varphi )+\mathcal{B}_{n}M(\varphi )=\varphi , & t\in (0,T),
\\ 
M(\varphi )(0)=0, & 
\end{array}%
\right.   \label{tilde M prob}
\end{equation}%
where $\mathcal{B}_{n}=(n-2)C_{0}^{-1}$, and $\varphi \in L^{2}(0,T)$ is
given. This operator is related to the region that is the complementary to
the one were the controls are localized. We also consider the adjoint
operator $M^{\ast }(\varphi )$ given by%
\begin{equation}
\left \{ 
\begin{array}{lr}
-\partial _{t}M^{\ast }(\varphi )+\mathcal{B}_{n}M^{\ast }(\varphi )=\varphi
, & t\in (0,T), \\ 
M^{\ast }(\varphi )(T)=0. & 
\end{array}%
\right.   \label{adj tilde M prob}
\end{equation}

For the domain that contains particles to which the controls are applied, we
introduce operator $G(\varphi )$ that satisfies the similar problem to %
\eqref{tilde M prob}, but with a different coefficient 
\begin{equation}
\left \{ 
\begin{array}{lr}
\partial _{t}G(\varphi )+(\mathcal{B}_{n}+N^{-1})G(\varphi )=\varphi , & 
t\in (0,T), \\ 
G(\varphi )(0)=0.\quad  & 
\end{array}%
\right.   \label{G prob 2}
\end{equation}%
This operator $G(\varphi )$ can be explicitly written as%
\begin{equation*}
G(\varphi )(t)=\int \limits_{0}^{t}e^{-(\mathcal{B}_{n}+N^{-1})(t-s)}\varphi
(s)ds,
\end{equation*}%
which show the non-local in time nature. We define its adjoint operator $%
G^{\ast }$ as the solution of the problem adjoint to \eqref{G prob 2} 
\begin{equation}
\left \{ 
\begin{array}{lr}
-\partial _{t}G^{\ast }(\varphi )+(\mathcal{B}_{n}+N^{-1})G^{\ast }(\varphi
)=\varphi , & t\in (0,T), \\ 
G^{\ast }(\varphi )(x,T)=0. & 
\end{array}%
\right.   \label{adj G prob 2}
\end{equation}%
Nevertheless, it turns out that for the case of boundary controls, as we are
assuming in problem \eqref{init state prob}, we will need to define some new
operators $H$, and $H^{\ast }$, coupled with $G^{\ast }$ and $G,$
respectively, in the following way 
\begin{equation}
\left \{ 
\begin{array}{lr}
-\partial _{t}H^{\ast }(\varphi )+(\mathcal{B}_{n}+N^{-1})H^{\ast }(\varphi
)-N^{-1}(\mathcal{B}_{n}+N^{-1})G(H^{\ast }(\varphi ))=\varphi , & t\in
(0,T), \\ 
H^{\ast }(\varphi )(T)=0, & 
\end{array}%
\right.   \label{adj H prob 2}
\end{equation}%
and 
\begin{equation}
\left \{ 
\begin{array}{lr}
\partial _{t}H(\varphi )+(\mathcal{B}_{n}+N^{-1})H(\varphi )-N^{-1}(\mathcal{%
B}_{n}+N^{-1})G^{\ast }(H(\varphi ))=\varphi , & t\in (0,T), \\ 
H(\varphi )(0)=0. & 
\end{array}%
\right.   \label{H prob 2}
\end{equation}%
Notice that both problems are now non-local in time, but since the operators 
$G$ and $G^{\ast }$ are globally Lipschitz continuous on $L^{2}(0,T)$, we
get the existence and uniqueness of the associate solutions once that $%
\varphi \in L^{2}(0,T)$ is given.


It is straightforward to show that the operator $G$ is the adjoint operator
to $G^{\ast }$, i.e. 
\begin{equation}
\int \limits_{0}^{T}G(\varphi )\psi {dt}=\int \limits_{0}^{T}\varphi G^{\ast
}(\psi ){dt},  \label{G adj rel}
\end{equation}%
for any arbitrary functions $\varphi ,\, \psi \in L^{2}(0,T)$. Indeed, we
have 
\begin{gather*}
\int \limits_{0}^{T}G(\varphi )\psi {dt}=\int \limits_{0}^{T}G(\varphi
)(-\partial _{t}G^{\ast }(\psi )+(\mathcal{B}_{n}+N^{-1})G^{\ast }(\psi )){dt%
} \\
=\int \limits_{0}^{T}(\mathcal{B}_{n}+N^{-1})G(\varphi )G^{\ast }(\psi ){dt}%
-G(\varphi )G^{\ast }(\varphi )\Big \vert_{0}^{T}+\int
\limits_{0}^{T}\partial _{t}G(\varphi )G^{\ast }(\psi ){dt} \\
=\int \limits_{0}^{T}(\partial _{t}G(\varphi )+(\mathcal{B}%
_{n}+N^{-1})G(\varphi ))G^{\ast }(\psi ){dt}=\int \limits_{0}^{T}\varphi
G^{\ast }(\psi ){dt}.
\end{gather*}%
Similarly to the above argument we know that $M$ is the adjoint operator to $%
M^{\ast }$%
\begin{equation*}
\int \limits_{0}^{T}M(\varphi )\psi {dt}=\int \limits_{0}^{T}\varphi M^{\ast
}(\psi ){dt}.
\end{equation*}

The case of operators $H$ and $H^{\ast }$ is less trivial. Nevertheless, we
also have that, for any arbitrary functions $\varphi ,\, \psi \in L^{2}(0,T),
$ 
\begin{equation}
\int \limits_{0}^{T}H(\varphi )\psi {dt}=\int \limits_{0}^{T}\varphi H^{\ast
}(\psi ){dt}.  \label{H adj rel}
\end{equation}%
Indeed, we make the following transformations 
\begin{gather*}
\int \limits_{0}^{T}H(\varphi )\psi {dt}=\int \limits_{0}^{T}H(\varphi
)(-\partial _{t}H^{\ast }(\psi )+(\mathcal{B}_{n}+N^{-1})H^{\ast }(\psi
)-N^{-1}(\mathcal{B}_{n}+N^{-1})G(H^{\ast }(\psi ))){dt} \\
=\int \limits_{0}^{T}\partial _{t}H(\varphi )H^{\ast }(\psi ){dt}+(\mathcal{B%
}_{n}+N^{-1})\int \limits_{0}^{T}H(\varphi )H^{\ast }(\psi ){dt} \\
-N^{-1}(\mathcal{B}_{n}+N^{-1})\int \limits_{0}^{T}G^{\ast }(H(\varphi
))H^{\ast }(\psi ){dt} \\
=\int \limits_{0}^{T}(\partial _{t}H(\varphi )+(\mathcal{B}%
_{n}+N^{-1})H(\varphi )-N^{-1}(\mathcal{B}_{n}+N^{-1})G^{\ast }(H(\varphi
)))H^{\ast }(\psi ){dt}=\int \limits_{0}^{T}\varphi H^{\ast }(\psi ){dt.}
\end{gather*}

In addition to the above properties it will be useful to get some other
relations among the above operators.

\begin{lemma}
\label{lem:aux relations} For the functions $G$, $G^{\ast }$, $H$ and $%
H^{\ast }$ introduced in \eqref{G prob 2}- \eqref{adj H prob 2}, we have the
following relations

i) $G^{\ast }(H(\varphi ))=H^{\ast }(G(\varphi ))$, and $G(H^{\ast }(\varphi
))=H(G^{\ast }(\varphi ))$,

ii) $H(\varphi )=G(\varphi )+N^{-1}(\mathcal{B}_{n}+N^{-1})G(H^{\ast
}(G(\varphi )))$, and we also have the adjoint version $H^{\ast }(\varphi
)=G^{\ast }(\varphi )+N^{-1}(\mathcal{B}_{n}+N^{-1})G^{\ast }(H(G^{\ast
}(\varphi )))$.
\end{lemma}

\begin{proof}
We start with the proof of \textit{i)} (the relations given in \textit{ii)}
are the direct consequences of the given in i)). We consider two coupled
auxiliary linear systems. The first one is a system coupling the functions $%
A(t)=G^{\ast }(H(\varphi ))$ and $B(t)=H(\varphi )$. From \eqref{adj G prob
2} and \eqref{H prob 2}, we get 
\begin{equation}
\left \{ 
\begin{array}{ll}
-\partial _{t}A+(\mathcal{B}_{n}+N^{-1})A=B, & t\in (0,T), \\ 
\partial _{t}B+(\mathcal{B}_{n}+N^{-1})B-N^{-1}(\mathcal{B}%
_{n}+N^{-1})A=\varphi , & t\in (0,T), \\ 
A(T)=0,\quad B(0)=0. & 
\end{array}%
\right.  \label{lem:func sys 1}
\end{equation}%
Analogously, from \eqref{G prob 2} and \eqref{adj H prob 2}, we get a second
system coupling the functions, $\tilde{A}(t)=G(H^{\ast }(G(\varphi )))$ and $%
\tilde{B}(t)=H^{\ast }(G(\varphi ))$: 
\begin{equation}
\left \{ 
\begin{array}{ll}
\partial _{t}\tilde{A}+(\mathcal{B}_{n}+N^{-1})\tilde{A}=\tilde{B}, & t\in
(0,T), \\ 
-\partial _{t}\tilde{B}+(\mathcal{B}_{n}+N^{-1})\tilde{B}-N^{-1}(\mathcal{B}%
_{n}+N^{-1})\tilde{A}=G(\varphi ), & t\in (0,T), \\ 
\tilde{A}(0)=0,\quad \tilde{B}(T)=0. & 
\end{array}%
\right.  \label{lem:func sys 2}
\end{equation}

From the system~\eqref{lem:func sys 1}, we can derive a linear second order
ODE problem on the function $A$. To do this, we substitute the expression
for $B$ from the first equation of the system \eqref{lem:func sys 1} into
the second one. Thus, we get 
\begin{equation*}
-\partial _{tt}^{2}A+(\mathcal{B}_{n}+N^{-1})\partial _{t}A-(\mathcal{B}%
_{n}+N^{-1})\partial _{t}A+(\mathcal{B}_{n}+N^{-1})^{2}A-N^{-1}(\mathcal{B}%
_{n}+N^{-1})A=\varphi ,
\end{equation*}%
and, simplifying, we get 
\begin{equation}
-\partial _{tt}^{2}A+\mathcal{B}_{n}(\mathcal{B}_{n}+N^{-1})A=\varphi .
\label{second order ODE}
\end{equation}%
Also, substituting $B$ written in terms of $A$ into $B(0)=0$, we get a
condition on $A(0)$ (recall that we already have the condition at $t=T$ from
the definition of $A$). Hence, the two boundary conditions are 
\begin{equation}
A(T)=0,\quad -\partial _{t}A(0)+(\mathcal{B}_{n}+N^{-1})A(0)=0.
\label{bound cond ODE}
\end{equation}%
This is a linear coercive equation which have uniqueness of solutions. For
instance, if we consider the homogeneous case ($\varphi \equiv 0$) in the
equation \eqref{second order ODE} we get that obviously, the trivial
solution satisfies this problem. Moreover, by multiplying the equation by $A$
and integrating from $0$ to $T$, we get 
\begin{equation*}
\int \limits_{0}^{T}|\partial _{t}A|^{2}{dt}+\mathcal{B}_{n}(\mathcal{B}%
_{n}+N^{-1})\int \limits_{0}^{T}A^{2}{dt}+(\mathcal{B}_{n}+N^{-1})A^{2}(0)=0,
\end{equation*}%
and we immediately conclude that, necessarily, $A\equiv 0$. Thus, for any $%
\varphi \in L^{2}(0,T)$, the inhomogeneous problem has also a unique
solution.

Let us obtain now the ODE satisfied by the function $\tilde{B}(t)$. From the
second equation of \eqref{lem:func sys 2}, we write $\tilde{A}$ in terms of $%
\tilde{B}$ 
\begin{equation*}
\tilde{A}=-N(\mathcal{B}_{n}+N^{-1})^{-1}\partial _{t}\tilde{B}+N\tilde{B}-N(%
\mathcal{B}_{n}+N^{-1})^{-1}G(\varphi ).
\end{equation*}%
Substituting this expression into the first equation of 
\eqref{lem:func sys
2}, we derive 
\begin{gather*}
-N(\mathcal{B}_{n}+N^{-1})^{-1}\partial _{tt}^{2}\tilde{B}+N\partial _{t}%
\tilde{B}-N(\mathcal{B}_{n}+N^{-1})^{-1}\partial _{t}G(\varphi ) \\
-N\partial _{t}\tilde{B}+(\mathcal{B}_{n}+N^{-1})N\tilde{B}-NG(\varphi )=%
\tilde{B}.
\end{gather*}%
Combining similar terms and using the definition of $G(\varphi )$, we get 
\begin{equation*}
-\partial _{tt}^{2}\tilde{B}+\mathcal{B}_{n}(\mathcal{B}_{n}+N^{-1})\tilde{B}%
=\partial _{t}G(\varphi )+(\mathcal{B}_{n}+N^{-1})G(\varphi )=\varphi .
\end{equation*}%
From condition $\tilde{A}(0)=0$, we conclude that 
\begin{equation*}
-\partial _{t}\tilde{B}(0)+(\mathcal{B}_{n}+N^{-1})\tilde{B}(0)=0,\quad 
\tilde{B}(T)=0.
\end{equation*}%
Therefore, we get exactly the same linear problem as for the function $A$.
But, since the solution to this problem is unique, we get that $A=\tilde{B},$
or in other terms $G^{\ast }(H(\varphi ))=H^{\ast }(G(\varphi ))$. This
concludes the proof of the first relation in \textit{(i)}.

The part \textit{(ii)} can be proved using some similar arguments. Using the
definition of $H^{\ast }$, we have that 
\begin{equation*}
G(\varphi )+N^{-1}(\mathcal{B}_{n}+N^{-1})G(H^{\ast }(G(\varphi
)))=-\partial _{t}H^{\ast }(G(\varphi ))+(\mathcal{B}_{n}+N^{-1})H^{\ast
}(G(\varphi )).
\end{equation*}%
We use \textit{(i)} and substitute $H^{\ast }(G(\varphi ))$ with $G^{\ast
}(H(\varphi ))$ in the right-hand side, and using the definition of $G^{\ast
}$, we get 
\begin{equation*}
G(\varphi )+N^{-1}(\mathcal{B}_{n}+N^{-1})G(H^{\ast }(G(\varphi
)))=-\partial _{t}G^{\ast }(H(\varphi ))+(\mathcal{B}_{n}+N^{-1})G^{\ast
}(H(\varphi ))=H(\varphi ).
\end{equation*}%
The second relation in is proved in a similar way. This concludes the proof.
\end{proof}

\section{Statement of the homogenization theorems}

Now, we are in a position to state the main theorem that characterizes the
pair of functions $(u_0, p_0)$ given by \eqref{limit func def}.

\begin{theorem}
\label{homogenization theorem} Let $n\geq 3$, $a_{\varepsilon
}=C_{0}\varepsilon ^{\gamma }$, where $C_{0}>0$, $\gamma =n/(n-2)$. If the
pair $(u_{\varepsilon },p_{\varepsilon })$ is the solution to the problem 
\eqref{init
coupled prob}, then, the pair $(u_{0},p_{0})$ is a solution to the system 
\begin{equation}
\left \{ 
\begin{array}{lr}
\partial _{t}u_{0}-\Delta u_{0}+\mathcal{A}_{n}(u_{0}-\mathcal{B}%
_{n}H(u_{0}))\chi _{\omega ^{T}} &  \\ 
+\mathcal{A}_{n}(u_{0}-\mathcal{B}_{n}M(u_{0}))\chi _{(\Omega \setminus 
\overline{\omega })\times (0,T)}=f-N^{-1}\mathcal{A}_{n}\mathcal{B}%
_{n}H(G^{\ast }(p_{0}))\chi _{\omega ^{T}}, & (x,t)\in Q^{T}, \\ 
-\partial _{t}p_{0}-\Delta p_{0} &  \\ 
+\mathcal{A}_{n}(p_{0}-\mathcal{B}_{n}H^{\ast }(p_{0}))\chi _{\omega ^{T}}+%
\mathcal{A}_{n}(p_{0}-\mathcal{B}_{n}M^{\ast }(p_{0}))\chi _{(\Omega
\setminus \overline{\omega })\times (0,T)} &  \\ 
=-\Delta u_{0}+\mathcal{A}_{n}(u_{0}-\mathcal{B}_{n}(\mathcal{B}%
_{n}+N^{-1})G^{\ast }(H(u_{0}))\chi _{\omega ^{T}} &  \\ 
+\mathcal{A}_{n}(u_{0}-\mathcal{B}_{n}^{2}M^{\ast }(M(u_{0})))\chi _{(\Omega
\setminus \overline{\omega })\times (0,T)}, & (x,t)\in Q^{T} \\ 
u_{0}(x,0)=0, & x\in \Omega , \\ 
u_{0}(x,t)=0, & (x,t)\in \Gamma ^{T}, \\ 
p(x,T)=u_{0}(x,T), & x\in \Omega , \\ 
p(x,t)=0, & (x,t)\in \Gamma ^{T}.%
\end{array}%
\right.  \label{homogenized system}
\end{equation}%
Moreover, if we introduce the limit state problem 
\begin{equation}
\left \{ 
\begin{array}{lr}
\partial _{t}u_{0}(v)-\Delta u_{0}(v)+\mathcal{A}_{n}(u_{0}(v)-\mathcal{B}%
_{n}H(u_{0}(v)))\chi _{\omega ^{T}} &  \\ 
+\mathcal{A}_{n}(u_{0}(v)-\mathcal{B}_{n}M(u_{0}(v)))\chi _{(\Omega
\setminus \overline{\omega })\times (0,T)}=f+\mathcal{A}_{n}\mathcal{B}%
_{n}v\chi _{\omega ^{T}}, & (x,t)\in Q^{T}, \\ 
u_{0}(v)(x,0)=0, & x\in \Omega , \\ 
u_{0}(v)(x,t)=0, & (x,t)\in \Gamma ^{T},%
\end{array}%
\right.  \label{limit state problem func}
\end{equation}%
where $v\in H^{1}(0,T;L^{2}(\omega ))$, with $v(x,0)=0,$ and if we define
the cost functional $J_{0}(v)$ given by (\ref{limit functional}),\ then,
from \eqref{limit state problem func} and \eqref{limit func def}, we have
that $u_{0}(v)$ is the state associated to the optimal control problem 
\begin{equation}
J_{0}(v_{0})=\min \limits_{v\in U_{ad}}J_{0}(v),  \label{limit opt cont prob}
\end{equation}%
where the set of admissible functions is given by 
\begin{equation*}
U_{ad}=\{ \psi \in H^{1}(0,T;L^{2}(\omega ))|\text{ }\psi (x,0)=0\}.
\end{equation*}
\end{theorem}

In Section 7 we will prove the convergence of the sequence of functionals $%
J_{\varepsilon }$ to the limit functional $J_{0}$ given by (\ref{limit
functional}). This will show that the system \eqref{homogenized
system} characterizes the optimal control problem \eqref{limit opt cont prob}%
, i.e. that $v_{0}=-N^{-1}H(G^{\ast }(p_{0}))\chi _{\omega ^{T}}$. Then we
will conclude

\begin{theorem}
\label{thm: cost func limit} Under the conditions of Theorem~\ref%
{homogenization theorem}, we have 
\begin{equation*}
\lim \limits_{\varepsilon \rightarrow 0}J_{\varepsilon }(v_{\varepsilon
})=J_{0}(v_{0}),
\end{equation*}%
where $v_{\varepsilon }$ is the optimal control of the 
\eqref{init state
prob}, \eqref{cost functional}, and $v_{0}$ is the optimal control of the
limit optimal control problem \eqref{limit opt cont prob}. 
\end{theorem}

\bigskip Lastly, we will show that the system \eqref{homogenized system} is
related to the limit functional and the limit optimal control $v_{0}$.

\begin{theorem}
\label{Thm convergence controls}Let the pair of functions $%
(u_{0}(v_{0}),v_{0})$ be an optimal solution of the problem 
\eqref{limit opt
cont prob}, then $v_{0}=-N^{-1}H(G^{\ast }(p_{0}))\chi _{\omega ^{T}}$,
where $p_{0}$ is the solution to \eqref{p_0 limit prob}.
\end{theorem}

\section{Proof of the homogenization theorem}

\begin{proof}
of Theorem \ref{homogenization theorem}. The main idea is to adapt to our
setting the great lines of the so called \textit{alternating test functions }%
(initially due to Luc Tartar to some simple framework and then extended by
many different authors: see, e.g., the monograph \cite{DiGoShBook}). We
introduce auxiliary functions $w_{\varepsilon }^{j}$, $j\in \mathbb{Z}^{n}$,
that are solutions to the boundary-value problems 
\begin{equation}
\left \{ 
\begin{array}{lr}
\Delta {w}_{\varepsilon }^{j}=0, & x\in T_{{\varepsilon }/4}^{j}\setminus 
\overline{G_{\varepsilon }^{j}}, \\ 
w_{\varepsilon }^{j}=1, & x\in \partial {G_{\varepsilon }^{j}}, \\ 
w_{\varepsilon }^{j}=0, & x\in \partial {T_{{\varepsilon }/4}^{j}},%
\end{array}%
\right.  \label{def: auxiliary func w}
\end{equation}%
where $T_{{\varepsilon }/4}^{j}$ denotes the ball centered in $%
P_{\varepsilon }^{j}$ of ${{\varepsilon }/4}$ radii. Based on these
functions, we construct auxiliary functions in the whole domain $\Omega $
(where $i=1,2$) 
\begin{equation}
W_{i,\varepsilon }=\left \{ 
\begin{array}{lr}
w_{\varepsilon }^{j}(x), & x\in T_{\varepsilon /4}^{j}\setminus \overline{%
G_{\varepsilon }^{j}},\,j\in \Upsilon _{\varepsilon }^{i}, \\ 
1, & x\in \overline{G_{\varepsilon }^{j}},\,j\in \Upsilon _{\varepsilon
}^{i}, \\ 
0, & x\in \Omega \setminus \overline{\bigcup \limits_{j\in \Upsilon
_{\varepsilon }^{i}}{T_{{\varepsilon }/4}^{j}}}.%
\end{array}%
\right.  \label{def: W aux func}
\end{equation}%
They are related to controllable and uncontrollable sets of particles. Note
that $W_{i,\varepsilon }\in H_{0}^{1}(\Omega )$ and $W_{i,\varepsilon
}\rightharpoonup 0$ weakly in $H_{0}^{1}(\Omega )$ as $\varepsilon
\rightarrow 0$. Due to the embedding theorems, for some subsequence for
which we preserve the notation of the original, we have $W_{i,\varepsilon
}\rightarrow 0$ strongly in $L^{2}(\Omega )$ as $\varepsilon \rightarrow 0$.

We will structure this long proof in a series of different steps.

\textit{Step A.1.} We take $W_{2,\varepsilon }H^{\ast }(\varphi )$, where $%
\varphi =\psi (x)\eta (t)$ with $\psi (x)\in C_{0}^{\infty }(\Omega )$, $%
\eta (t)\in C^{1}([0,T])$, as a test function in the integral identity %
\eqref{int ident u} and get 
\begin{gather}
\int \limits_{Q_{\varepsilon }^{T}}\partial _{t}u_{\varepsilon
}W_{2,\varepsilon }H^{\ast }(\varphi ){dxdt}+\varepsilon ^{-\gamma }\int
\limits_{S_{\varepsilon }^{2,T}}\partial _{t}u_{\varepsilon }H^{\ast
}(\varphi ){dsdt}+\int \limits_{Q_{\varepsilon }^{T}}\nabla u_{\varepsilon
}\nabla (W_{2,\varepsilon }H^{\ast }(\varphi )){dxdt}  \notag \\
=\int \limits_{Q_{\varepsilon }^{T}}fW_{2,\varepsilon }H^{\ast }(\varphi ){%
dxdt}-N^{-1}\int \limits_{S_{\varepsilon }^{2,T}}p_{\varepsilon }H^{\ast
}(\varphi ){dsdt}.  \label{thm:W2 u first}
\end{gather}%
Using the convergences \eqref{limit func def} and the properties of $%
W_{2,\varepsilon }$, we have 
\begin{gather*}
\lim \limits_{\varepsilon \rightarrow 0}\int \limits_{Q_{\varepsilon
}^{T}}\partial _{t}u_{\varepsilon }W_{2,\varepsilon }H^{\ast }(\varphi ){dxdt%
}=0,\quad \lim \limits_{\varepsilon \rightarrow 0}\int
\limits_{Q_{\varepsilon }^{T}}fW_{2,\varepsilon }H^{\ast }(\varphi ){dxdt}=0,
\\
\lim \limits_{\varepsilon \rightarrow 0}\int \limits_{Q_{\varepsilon
}^{T}}\nabla u_{\varepsilon }\nabla (W_{2,\varepsilon }H^{\ast }(\varphi )){%
dxdt}=\lim \limits_{\varepsilon \rightarrow 0}\int \limits_{Q_{\varepsilon
}^{T}}\nabla (u_{\varepsilon }H^{\ast }(\varphi ))\nabla W_{2,\varepsilon }{%
dxdt}.
\end{gather*}%
Thus, from \eqref{thm:W2 u first}, we derive 
\begin{equation*}
\varepsilon ^{-\gamma }\int \limits_{S_{\varepsilon }^{2,T}}\partial
_{t}u_{\varepsilon }H^{\ast }(\varphi ){dsdt}=-\int \limits_{Q_{\varepsilon
}^{T}}\nabla W_{2,\varepsilon }\nabla (u_{\varepsilon }H^{\ast }(\varphi )){%
dxdt}-N^{-1}\int \limits_{S_{\varepsilon }^{2,T}}p_{\varepsilon }H^{\ast
}(\varphi ){dsdt}+\zeta _{\varepsilon },
\end{equation*}%
where, here and below, $\zeta _{\varepsilon }\rightarrow 0$ as $\varepsilon
\rightarrow 0$ (we will abuse the notation and will always use $\zeta
_{\varepsilon }$ for the terms converging to zero).

Next, we use the following fundamental relation (calling in Theorem 4.5 of 
\cite{DiGoShBook} as \textquotedblleft from surface to volume averaging
convergence principle\textquotedblright \ and also applied, under different
formulations, by many different authors: see, e.g., \cite{OlSh95} and \cite%
{ZuShDiffEq}, among many others). We have 
\begin{equation}
\int \limits_{\Omega _{\varepsilon }}\nabla W_{i,\varepsilon }\nabla \eta {dx%
}=-\mathcal{A}_{n}\int \limits_{\Omega ^{i}}\eta {dx}+\mathcal{B}%
_{n}\varepsilon ^{-\gamma }\sum \limits_{j\in \Upsilon _{\varepsilon
}^{i}}\int \limits_{\partial G_{\varepsilon }^{j}}\eta {ds}+\zeta
_{\varepsilon },  \label{w limit relation}
\end{equation}%
here $\Omega ^{1}=\Omega \setminus \omega $, $\Omega ^{2}=\omega $, and then 
\begin{gather*}
\lim \limits_{\varepsilon \rightarrow 0}\varepsilon ^{-\gamma }\int
\limits_{S_{\varepsilon }^{2,T}}\partial _{t}u_{\varepsilon }H^{\ast
}(\varphi ){dsdt}=\mathcal{A}_{n}\int \limits_{\omega ^{T}}H(u_{0})\varphi {%
dxdt}-\lim \limits_{\varepsilon \rightarrow 0}\varepsilon ^{-\gamma }%
\mathcal{B}_{n}\int \limits_{S_{\varepsilon }^{2,T}}u_{\varepsilon }H^{\ast
}(\varphi ){dsdt} \\
-N^{-1}\lim \limits_{\varepsilon \rightarrow 0}\varepsilon ^{-\gamma }\int
\limits_{S_{\varepsilon }^{2,T}}p_{\varepsilon }H^{\ast }(\varphi ){dsdt}.
\end{gather*}%
Using that $u_{\varepsilon }(x,0)=0$ and $H^{\ast }(\varphi )(x,T)=0$, we
integrate by parts the integral in the right-hand side, and further
transform the previous equality to get 
\begin{gather}
\lim \limits_{\varepsilon \rightarrow 0}\varepsilon ^{-\gamma }\int
\limits_{S_{\varepsilon }^{2,T}}u_{\varepsilon }(-\partial _{t}H^{\ast
}(\varphi )+\mathcal{B}_{n}H^{\ast }(\varphi )){dsdt}  \notag \\
=\mathcal{A}_{n}\int \limits_{\omega ^{T}}H(u_{0})\varphi {dxdt}-\lim
\limits_{\varepsilon \rightarrow 0}N^{-1}\varepsilon ^{-\gamma }\int
\limits_{S_{\varepsilon }^{2,T}}p_{\varepsilon }H^{\ast }(\varphi ){dsdt}.
\label{thm: u limit rel 1}
\end{gather}

\textit{Step A.2. }We take $\varphi =W_{2,\varepsilon }G(\varphi )$, where $%
\varphi =\psi (x)\eta (t)$ with $\psi \in C_{0}^{\infty }(\Omega )$, $\eta
\in C^{1}([0,T])$ as a test function in the integral identity 
\eqref{int ident init adjoint
prob}, and get 
\begin{gather}
-\int \limits_{Q_{\varepsilon }^{T}}\partial _{t}p_{\varepsilon
}W_{2,\varepsilon }G(\varphi ){dxdt}-\int \limits_{S_{\varepsilon
}^{2,T}}\partial _{t}p_{\varepsilon }G(\varphi ){dsdt}+\int
\limits_{Q_{\varepsilon }^{T}}\nabla p_{\varepsilon }\nabla
(W_{2,\varepsilon } G(\varphi )){dxdt}  \notag \\
=\int \limits_{Q_{\varepsilon }^{T}}\nabla u_{\varepsilon }\nabla
(W_{2,\varepsilon }G(\varphi )){dxdt}.  \label{thm: W_2 p first}
\end{gather}%
Taking $W_{2,\varepsilon }G(\varphi )$ as a test function in 
\eqref{int
ident u}, we have (using the same arguments as a above) 
\begin{equation*}
\int \limits_{Q_{\varepsilon }^{T}}\nabla u_{\varepsilon }\nabla
(W_{2,\varepsilon }G(\varphi )){dxdt}=-\int \limits_{S_{\varepsilon
}^{2,T}}\partial _{t}u_{\varepsilon }G(\varphi ){dsdt}-N^{-1}\varepsilon
^{-\gamma }\int \limits_{S_{\varepsilon }^{2,T}}p_{\varepsilon }G(\varphi ){%
dsdt}+\zeta _{\varepsilon }.
\end{equation*}%
Substituting this relation into \eqref{thm: W_2 p first}, we derive 
\begin{gather}
\int \limits_{Q_{\varepsilon }^{T}}\nabla W_{2,\varepsilon }\nabla
(p_{\varepsilon }G(\varphi )){dxdt}-\varepsilon ^{-\gamma }\int
\limits_{S_{\varepsilon }^{2,T}}\partial _{t}(p_{\varepsilon
}-u_{\varepsilon })G(\varphi ){dsdt}  \notag \\
=\int \limits_{Q_{\varepsilon }^{T}}\partial _{t}p_{\varepsilon
}W_{2,\varepsilon }G(\varphi ){dxdt}-N^{-1}\varepsilon ^{-\gamma }\int
\limits_{S_{\varepsilon }^{2,T}}p_{\varepsilon }G(\varphi ){dsdt}+\zeta
_{\varepsilon }.  \label{thm: p ident G}
\end{gather}%
Using the properties of $W_{2,\varepsilon }$ and a priori estimates of $%
u_{\varepsilon }$ and $p_{\varepsilon }$, we conclude that the first
integral in the right-hand side of the above equality converge to zero as $%
\varepsilon \rightarrow 0$. Also, using the relation \eqref{w limit relation}%
, we derive 
\begin{gather}
-\varepsilon ^{-\gamma }\int \limits_{S_{\varepsilon }^{2,T}}\partial
_{t}(p_{\varepsilon }-u_{\varepsilon })G(\varphi ){dsdt}+\varepsilon
^{-\gamma }\mathcal{B}_{n}\int \limits_{S_{\varepsilon
}^{2,T}}p_{\varepsilon }G(\varphi ){dsdt}=  \notag \\
=\mathcal{A}_{n}\int \limits_{\omega ^{T}}p_{0}G(\varphi ){dxdt}%
-N^{-1}\varepsilon ^{-\gamma }\int \limits_{S_{\varepsilon
}^{2,T}}p_{\varepsilon }G(\varphi ){dsdt}+\zeta _{\varepsilon }.
\label{thm: p ident G 2}
\end{gather}%
Note, that $p_{\varepsilon }(x,T)=u_{\varepsilon }(x,T)$ and $G(\varphi
)(0)=0$, therefore, we have 
\begin{equation*}
-\varepsilon ^{-\gamma }\int \limits_{S_{\varepsilon }^{2,T}}\partial
_{t}(p_{\varepsilon }-u_{\varepsilon })G(\varphi ){dsdt}=\varepsilon
^{-\gamma }\int \limits_{S_{\varepsilon }^{2,T}}\partial _{t}G(\varphi
)(p_{\varepsilon }-u_{\varepsilon }){dsdt}.
\end{equation*}%
Thus, from \eqref{thm: p ident G 2}, we get 
\begin{equation*}
\varepsilon ^{-\gamma }\int \limits_{S_{\varepsilon }^{2,T}}p_{\varepsilon
}(\partial _{t}G(\varphi )+(\mathcal{B}_{n}+N^{-1})G(\varphi )){dsdt} \\
=\mathcal{A}_{n}\int \limits_{\omega ^{T}}p_{0}G(\varphi ){dxdt}+\varepsilon
^{-\gamma }\int \limits_{S_{\varepsilon }^{2,T}}u_{\varepsilon }\partial
_{t}G(\varphi ){dsdt}+\zeta _{\varepsilon }.
\end{equation*}%
Using the definition of $G(\varphi )$, we conclude 
\begin{equation}  \label{p limit with u}
\varepsilon ^{-\gamma }\int \limits_{S_{\varepsilon }^{2,T}}p_{\varepsilon
}\varphi {dsdt}=\mathcal{A}_{n}\int \limits_{\omega ^{T}}p_{0}G(\varphi ){%
dxdt}+\varepsilon ^{-\gamma }\int \limits_{S_{\varepsilon
}^{2,T}}u_{\varepsilon }(\varphi -(\mathcal{B}_{n}+N^{-1})G(\varphi )){dsdt}%
+\zeta _{\varepsilon }.
\end{equation}%
Then, we substitute \eqref{p limit with u} (with $\varphi =H^{\ast }(\varphi
)$) into \eqref{thm: u limit rel 1}, and get 
\begin{gather*}
\lim \limits_{\varepsilon \rightarrow 0}\varepsilon ^{-\gamma }\int
\limits_{S_{\varepsilon }^{2,T}}u_{\varepsilon }(-\partial _{t}H^{\ast
}(\varphi )+(\mathcal{B}_{n}+N^{-1})H^{\ast }(\varphi )-N^{-1}(\mathcal{B}%
_{n}+N^{-1})G(H^{\ast }(\varphi ))){dsdt}= \\
=\mathcal{A}_{n}\int \limits_{\omega ^{T}}H(u_{0})\varphi {dxdt}-\mathcal{A}%
_{n}N^{-1}\int \limits_{\omega ^{T}}p_{0}G(H^{\ast }(\varphi )){dxdt}.
\end{gather*}%
Using the definition of $H^{\ast }$, we derive 
\begin{equation}  \label{thm: u
limit rel}
\lim \limits_{\varepsilon \rightarrow 0}\varepsilon ^{-\gamma }\int
\limits_{S_{\varepsilon }^{2,T}}u_{\varepsilon }\varphi {dsdt}=\mathcal{A}%
_{n}\int \limits_{\omega ^{T}}H(u_{0})\varphi {dxdt} - \mathcal{A}%
_{n}N^{-1}\int \limits_{\omega ^{T}}H(G^{\ast }(p_{0}))\varphi {dxdt}.
\end{equation}

\textit{Step A.3. }Now, we can find the limit of the integrals taken over $%
S_{\varepsilon }^{2,T}$ in the integral identity~\eqref{int ident u}.
Indeed, we take $W_{2,\varepsilon }\varphi $ as a test function in %
\eqref{int ident u}, and get 
\begin{gather}
\varepsilon ^{-\gamma }\int \limits_{S_{\varepsilon }^{2,T}}\partial
_{t}u_{\varepsilon }\varphi {dsdt}+\varepsilon ^{-\gamma }N^{-1}\int
\limits_{S_{\varepsilon }^{2,T}}p_{\varepsilon }\varphi {dsdt}=-\int
\limits_{\Omega _{\varepsilon }^{T}}\nabla u_{\varepsilon }\nabla
(W_{2,\varepsilon }\varphi ){dxdt}+\int \limits_{Q_{\varepsilon
}^{T}}fW_{2,\varepsilon }\varphi {dxdt}  \notag \\
=\mathcal{A}_{n}\int \limits_{\omega ^{T}}u_{0}\varphi {dxdt}-\varepsilon
^{-\gamma }\mathcal{B}_{n}\int \limits_{S_{\varepsilon
}^{2,T}}u_{\varepsilon }\varphi {dsdt}+\zeta _{\varepsilon }  \notag \\
=\mathcal{A}_{n}\int \limits_{\omega ^{T}}u_{0}\varphi {dxdt}-\mathcal{A}_{n}%
\mathcal{B}_{n}\int \limits_{\omega ^{T}}H(u_{0})\varphi {dxdt}+\mathcal{A}%
_{n}\mathcal{B}_{n}N^{-1}\int \limits_{\omega ^{T}}H(G^{\ast
}(p_{0}))\varphi {dxdt}+\zeta _{\varepsilon }.  \label{pa_t u conv control}
\end{gather}%
Thus, we found the limit of the terms related to $S_{\varepsilon }^{2,T}$.

\textit{Step A.4.} To complete the derivation of the limit equation, we
proceed with the terms related to $S_{\varepsilon }^{1,T}$. We take $%
W_{1,\varepsilon }M^{\ast }(\varphi )$ as a test function in the integral
identity \eqref{int ident u}, and obtain (again using the properties of the
function $W_{1,\varepsilon }$, we conclude that the integral with $f$
converges to zero) 
\begin{equation*}
\int \limits_{Q_{\varepsilon }^{T}}\nabla W_{1,\varepsilon }\nabla
(u_{\varepsilon }M^{\ast }(\varphi )){dxdt}+\varepsilon ^{-\gamma }\int
\limits_{S_{\varepsilon }^{1,T}}\partial _{t}u_{\varepsilon }M^{\ast
}(\varphi ){dsdt}=\zeta _{\varepsilon },
\end{equation*}%
where $\zeta _{\varepsilon }\rightarrow 0$ as $\varepsilon \rightarrow 0$.
Using \eqref{w limit
relation}, we derive 
\begin{equation*}
\varepsilon ^{-\gamma }\mathcal{B}_{n}\int \limits_{S_{\varepsilon
}^{1,T}}u_{\varepsilon }M^{\ast }(\varphi ){dsdt}+\varepsilon ^{-\gamma
}\int \limits_{S_{\varepsilon }^{1,T}}\partial _{t}u_{\varepsilon }M^{\ast
}(\varphi ){dsdt}=\mathcal{A}_{n}\int \limits_{(\Omega \setminus \overline{%
\omega })\times (0,T)}u_{0}M^{\ast }(\varphi ){dxdt}+\zeta _{\varepsilon }.
\end{equation*}%
As $M^{\ast }(\varphi )(x,T)=0$ and $u_{\varepsilon }(x,0)=0$, we have 
\begin{equation*}
\varepsilon ^{-\gamma }\int \limits_{S_{\varepsilon }^{1,T}}\partial
_{t}u_{\varepsilon }M^{\ast }(\varphi ){dsdt}=-\varepsilon ^{-\gamma }\int
\limits_{S_{\varepsilon }^{1,T}}u_{\varepsilon }\partial _{t}M^{\ast
}(\varphi ){dsdt}.
\end{equation*}%
Thus, we derive 
\begin{equation*}
\varepsilon ^{-\gamma }\int \limits_{S_{\varepsilon }^{1,T}}u_{\varepsilon
}(-\partial _{t}M^{\ast }(\varphi )+\mathcal{B}_{n}M^{\ast }(\varphi )){dsdt}%
=\mathcal{A}_{n}\int \limits_{(\Omega \setminus \overline{\omega })\times
(0,T)}u_{0}M^{\ast }(\varphi ){dxdt}+\zeta _{\varepsilon }.
\end{equation*}%
Using the definition of $M^{\ast }(\varphi )$, we get 
\begin{equation}  \label{u conv uncont}
\varepsilon ^{-\gamma }\int \limits_{S_{\varepsilon }^{1,T}}u_{\varepsilon
}\varphi {dsdt}=\mathcal{A}_{n}\int \limits_{(\Omega \setminus \overline{%
\omega })\times (0,T)}M(u_{0})\varphi {dxdt}+\zeta _{\varepsilon }.
\end{equation}%
Now, we take $W_{1,\varepsilon }\varphi $ as a test function in the integral
identity \eqref{int ident u}, and using similar arguments as above, we
derive 
\begin{equation*}
\int \limits_{Q_{\varepsilon }^{T}}\nabla W_{1,\varepsilon }\nabla
(u_{\varepsilon }\varphi ){dxdt}+\varepsilon ^{-\gamma }\int
\limits_{S_{\varepsilon }^{1,T}}\partial _{t}u_{\varepsilon }\varphi {dsdt}%
=\zeta _{\varepsilon }.
\end{equation*}%
We transform this identity using relation \eqref{w limit relation} and
obtain 
\begin{gather}
\varepsilon ^{-\gamma }\int \limits_{S_{\varepsilon }^{1,T}}\partial
_{t}u_{\varepsilon }\varphi {dsdt}=\mathcal{A}_{n}\int \limits_{(\Omega
\setminus \overline{\omega })\times (0,T)}u_{0}\varphi {dxdt}-\mathcal{B}%
_{n}\varepsilon ^{-\gamma }\int \limits_{S_{\varepsilon
}^{1,T}}u_{\varepsilon }\varphi {dsdt}+\zeta _{\varepsilon }  \notag \\
=\mathcal{A}_{n}\int \limits_{(\Omega \setminus \overline{\omega })\times
(0,T)}(u_{0}-\mathcal{B}_{n}M(u_{0}))\varphi {dxdt}+\zeta _{\varepsilon }.
\label{pa_t u conv uncontrol}
\end{gather}

Finally, using the convergences \eqref{pa_t u conv control} and 
\eqref{pa_t
u conv uncontrol}, we can pass to the limit in the integral identity for $%
u_{\varepsilon }$ and get the integral identity for $u_{0}$ 
\begin{gather}
\int \limits_{Q^{T}}\partial _{t}u_{0}\varphi {dxdt}+\int
\limits_{Q^{T}}\nabla u_{0}\nabla \varphi {dxdt}  \notag \\
+\mathcal{A}_{n}\int \limits_{(\Omega \setminus \overline{\omega })\times
(0,T)}(u_{0}-\mathcal{B}_{n}M(u_{0}))\varphi {dxdt}+\mathcal{A}_{n}\int
\limits_{w^{T}}(u_{0}-\mathcal{B}_{n}H(u_{0}))\varphi{dxdt}  \notag \\
=\int \limits_{Q^{T}}f\varphi {dxdt}-N^{-1}\mathcal{A}_{n}\mathcal{B}%
_{n}\int \limits_{w^{T}}H(G^{\ast }(p_{0}))\varphi {dxdt}.
\label{u_0 int ident}
\end{gather}%
Thus, $u_{0}$ is a solution to the following problem 
\begin{equation*}
\left \{ 
\begin{array}{ll}
\partial _{t}u_{0}-\Delta u_{0}+\mathcal{A}_{n}(u_{0}-\mathcal{B}%
_{n}H(u_{0}))\chi _{\omega ^{T}} &  \\ 
+\mathcal{A}_{n}(u_{0}-\mathcal{B}_{n}M(u_{0}))\chi _{(\Omega \setminus 
\overline{\omega })\times (0,T)}=f-N^{-1}\mathcal{A}_{n}\mathcal{B}%
_{n}H(G^{\ast }(p_{0}))\chi _{\omega ^{T}}, & (x,t)\in Q^{T}, \\ 
u_{0}(x,0)=0, & x\in \Omega , \\ 
u_{0}(x,t)=0, & (x,t)\in \Gamma ^{T}.%
\end{array}%
\right.
\end{equation*}

\textit{Step B.1. }Let us find the limit equation for $p_{0}$. We take $%
W_{2,\varepsilon }\varphi $ as a test function in the adjoint problem's
integral identity~\eqref{int
ident init adjoint prob}, and get 
\begin{equation*}
-\int \limits_{Q_{\varepsilon }^{T}}\partial _{t}p_{\varepsilon
}W_{2,\varepsilon }\varphi {dxdt}-\varepsilon ^{-\gamma }\int
\limits_{S_{\varepsilon }^{2,T}}\partial _{t}p_{\varepsilon }\varphi {dsdt}%
+\int \limits_{Q_{\varepsilon }^{T}}\nabla p_{\varepsilon }\nabla
(W_{2,\varepsilon }\varphi ){dxdt}=\int \limits_{Q_{\varepsilon }^{T}}\nabla
u_{\varepsilon }\nabla (W_{2,\varepsilon }\varphi ){dxdt}.
\end{equation*}%
The first integral in the left-hand side converges to zero due to the
properties of $W_{2,\varepsilon }$, thus, we have 
\begin{equation*}
-\varepsilon ^{-\gamma }\int \limits_{S_{\varepsilon }^{2,T}}\partial
_{t}p_{\varepsilon }\varphi {dsdt}=\int \limits_{Q_{\varepsilon }^{T}}\nabla
(u_{\varepsilon }-p_{\varepsilon })\nabla (W_{2,\varepsilon }\varphi ){dxdt}%
+\zeta _{\varepsilon },
\end{equation*}%
Using \eqref{w limit relation}, \eqref{p limit with u} and 
\eqref{thm: u
limit rel}, we derive 
\begin{gather*}
-\varepsilon ^{-\gamma }\int \limits_{S_{\varepsilon }^{2,T}}\partial
_{t}p_{\varepsilon }\varphi {dsdt}=-\mathcal{A}_{n}\int%
\limits_{w^{T}}(u_{0}-p_{0})\varphi {dxdt}+\varepsilon ^{-\gamma }\mathcal{B}%
_{n}\int \limits_{S_{\varepsilon }^{2,T}}(u_{\varepsilon }-p_{\varepsilon
})\varphi {dsdt} + \zeta_\varepsilon \\
=-\mathcal{A}_{n}\int \limits_{w^{T}}(u_{0}-p_{0})\varphi {dxdt}-\mathcal{A}%
_{n}\mathcal{B}_{n}\int \limits_{w^{T}}p_{0}G(\varphi ){dxdt}+\varepsilon
^{-\gamma }\mathcal{B}_{n}(\mathcal{B}_{n}+N^{-1})\int
\limits_{S_{\varepsilon }^{2,T}}u_{\varepsilon }G(\varphi ){dsdt} +
\zeta_\varepsilon \\
=\mathcal{A}_{n}\int \limits_{w^{T}}p_{0}(\varphi -\mathcal{B}_{n}G(\varphi
)){dxdt}-\mathcal{A}_{n}\int \limits_{w^{T}}u_{0}\varphi {dxdt}+\mathcal{A}%
_{n}\mathcal{B}_{n}(\mathcal{B}_{n}+N^{-1})\int
\limits_{Q^{T}}H(u_{0})G(\varphi ){dxdt} \\
-N^{-1}\mathcal{A}_{n}\mathcal{B}_{n}(\mathcal{B}_{n}+N^{-1})\int%
\limits_{w^{T}}p_{0}G(H^{\ast }(G(\varphi ))){dxdt} + \zeta_\varepsilon.
\end{gather*}%
Combining the terms with $p_{0}$, we get the expression 
\begin{equation*}
\varphi -\mathcal{B}_{n}G(\varphi )-N^{-1}\mathcal{B}_{n}(\mathcal{B}%
_{n}+N^{-1})G(H^{\ast }(G(\varphi )))\equiv E(\varphi ),
\end{equation*}%
Using relation \textit{(ii)} from Lemma~\ref{lem:aux relations}, we derive 
\begin{equation*}
E(\varphi )=\varphi -\mathcal{B}_{n}H(\varphi ).
\end{equation*}%
Therefore, the term with $p_{0}$ is 
\begin{equation*}
\mathcal{A}_{n}\int \limits_{\omega ^{T}}p_{0}(\varphi -\mathcal{B}%
_{n}H(\varphi )){dxdt}=\mathcal{A}_{n}\int \limits_{\omega ^{T}}(p_{0}-%
\mathcal{B}_{n}H^{\ast }(p_{0}))\varphi {dxdt}.
\end{equation*}%
Putting it all together, we derive 
\begin{gather}
-\varepsilon ^{-\gamma }\int \limits_{S_{\varepsilon }^{2,T}}\partial
_{t}p_{\varepsilon }\varphi {dsdt}=\mathcal{A}_{n}\int \limits_{w^{T}}(p_{0}-%
\mathcal{B}_{n}H^{\ast }(p_{0}))\varphi {dxdt}-  \notag \\
-\mathcal{A}_{n}\int \limits_{w^{T}}(u_{0}-\mathcal{B}_{n}(\mathcal{B}%
_{n}+N^{-1})H^{\ast }(G(u_{0})))\varphi {dxdt}.  \label{pa_t p cont}
\end{gather}

\textit{Step B.2. }Next, we deal with the parts related to $S_{\varepsilon
}^{1,T}$. We take $W_{1,\varepsilon }M(\varphi )$ as a test function in the
integral identity \eqref{int ident init adjoint prob}, and get 
\begin{gather*}
-\int \limits_{Q_{\varepsilon }^{T}}\partial _{t}p_{\varepsilon
}W_{1,\varepsilon }M(\varphi ){dxdt}-\varepsilon ^{-\gamma }\int
\limits_{S_{\varepsilon }^{1,T}}\partial _{t}p_{\varepsilon }M(\varphi ){dsdt%
}+\int \limits_{Q_{\varepsilon }^{T}}\nabla W_{1,\varepsilon }\nabla
(p_{\varepsilon }M(\varphi )){dxdt}= \\
=\int \limits_{Q_{\varepsilon }^{T}}\nabla W_{1,\varepsilon }\nabla
(u_{\varepsilon }M(\varphi ))dxdt + \zeta _{\varepsilon }.
\end{gather*}%
Using integral identity \eqref{int ident u} taken with the same test
function, we transform the last relation to the form 
\begin{gather*}
\int \limits_{Q_{\varepsilon }^{T}}\nabla W_{1,\varepsilon }\nabla
(p_{\varepsilon }M(\varphi )){dxdt}-\varepsilon ^{-\gamma }\int
\limits_{S_{\varepsilon }^{1,T}}\partial _{t}(p_{\varepsilon
}-u_{\varepsilon })M(\varphi ){dsdt}= \\
=\int \limits_{Q_{\varepsilon }^{T}}\partial _{t}(p_{\varepsilon
}-u_{\varepsilon })W_{1,\varepsilon }M(\varphi ){dxdt}+\int
\limits_{Q_{\varepsilon }^{T}}fW_{1,\varepsilon }M(\varphi ){dxdt}+\zeta
_{\varepsilon }.
\end{gather*}%
Properties of $W_{1,\varepsilon }$ imply that the terms at the right-hand
side converges to zero as $\varepsilon \rightarrow 0$. We use 
\eqref{w limit
relation} and transform the first term in the left-hand side and finally get 
\begin{equation*}
-\varepsilon ^{-\gamma }\int \limits_{S_{\varepsilon }^{1,T}}\partial
_{t}(p_{\varepsilon }-u_{\varepsilon })M(\varphi ){dsdt}+\varepsilon
^{-\gamma }\int \limits_{S_{\varepsilon }^{1,T}}p_{\varepsilon }M(\varphi ){%
dsdt}=\mathcal{A}_{n}\int \limits_{(\Omega \setminus \overline{\omega }%
)\times (0,T)}p_{0}M(\varphi ){dxdt}+\zeta _{\varepsilon }.
\end{equation*}%
Integrating by parts in the first term, we conclude 
\begin{equation*}
\varepsilon ^{-\gamma }\int \limits_{S_{\varepsilon }^{1,T}}(p_{\varepsilon
}-u_{\varepsilon })\partial _{t}M(\varphi ){dsdt}+\varepsilon ^{-\gamma
}\int \limits_{S_{\varepsilon }^{1,T}}p_{\varepsilon }M(\varphi ){dsdt}=%
\mathcal{A}_{n}\int \limits_{(\Omega \setminus \overline{\omega })\times
(0,T)}p_{0}M(\varphi ){dxdt}+\zeta _{\varepsilon }.
\end{equation*}%
Using the definition of $M^{\ast }$ and the convergence \eqref{u conv uncont}%
, we derive 
\begin{gather}
\varepsilon ^{-\gamma }\int \limits_{S_{\varepsilon }^{1,T}}p_{\varepsilon
}\varphi {dsdt}=\varepsilon ^{-\gamma }\int \limits_{S_{\varepsilon
}^{1,T}}p_{\varepsilon }(\partial _{t}M(\varphi )+\mathcal{B}_{n}M(\varphi ))%
{dsdt}  \notag \\
=\mathcal{A}_{n}\int \limits_{(\Omega \setminus \overline{\omega })\times
(0,T)}p_{0}M(\varphi ){dxdt}+\mathcal{A}_{n}\int \limits_{(\Omega \setminus 
\overline{\omega })\times (0,T)}M(u_{0})(\varphi -\mathcal{B}_{n}M(\varphi ))%
{dxdt}+\zeta _{\varepsilon }.  \label{p W_1 relation}
\end{gather}%
Lastly, we take $W_{1,\varepsilon }\varphi $ as a test function in the
integral identity \eqref{int ident init adjoint prob}, and using 
\eqref{pa_t
u conv uncontrol} and \eqref{p W_1 relation}, we get 
\begin{gather*}
-\varepsilon ^{-\gamma }\int \limits_{S_{\varepsilon }^{1,T}}\partial
_{t}p_{\varepsilon }\varphi {dsdt}=\int \limits_{Q_{\varepsilon }^{T}}\nabla
(u_{\varepsilon }-p_{\varepsilon })\nabla (W_{1,\varepsilon }\varphi ){dxdt}%
+\zeta _{\varepsilon } \\
=\mathcal{A}_{n}\int \limits_{(\Omega \setminus \overline{\omega })\times
(0,T)}(p_{0}-u_{0})\varphi {dxdt}+\varepsilon ^{-\gamma }\mathcal{B}_{n}\int
\limits_{S_{\varepsilon }^{1,T}}(u_{\varepsilon }-p_{\varepsilon })\varphi {%
dsdt}+\zeta _{\varepsilon } \\
=\mathcal{A}_{n}\int \limits_{(\Omega \setminus \overline{\omega })\times
(0,T)}(p_{0}-u_{0})\varphi {dxdt}+\mathcal{A}_{n}\mathcal{B}_{n}\int
\limits_{(\Omega \setminus \overline{\omega })\times (0,T)}M(u_{0})\varphi {%
dxdt} \\
-\mathcal{A}_{n}\mathcal{B}_{n}\int \limits_{(\Omega \setminus \overline{%
\omega })\times (0,T)}p_{0}M(\varphi ){dxdt}-\mathcal{A}_{n}\mathcal{B}%
_{n}\int \limits_{(\Omega \setminus \overline{\omega })\times
(0,T)}M(u_{0})(\varphi -\mathcal{B}_{n}M(\varphi )){dxdt}+\zeta
_{\varepsilon }.
\end{gather*}%
Grouping similar terms, we derive 
\begin{gather}
-\varepsilon ^{-\gamma }\int \limits_{S_{\varepsilon }^{1,T}}\partial
_{t}p_{\varepsilon }\varphi {dsdt}=\mathcal{A}_{n}\int \limits_{(\Omega
\setminus \overline{\omega })\times (0,T)}(p_{0}-\mathcal{B}_{n}M^{\ast
}(p_{0}))\varphi {dxdt}  \notag \\
-\mathcal{A}_{n}\int \limits_{(\Omega \setminus \overline{\omega })\times
(0,T)}(u_{0}-\mathcal{B}_{n}^{2}M^{\ast }(M(u_{0})))\varphi {dxdt}+\zeta
_{\varepsilon }.  \label{pa_t p uncont}
\end{gather}

\textit{Step B.3. }Now, using convergences \eqref{pa_t p cont} and %
\eqref{pa_t p uncont}, we can pass to the limit as $\varepsilon \rightarrow
0 $ in \eqref{int ident init adjoint
prob}, and get the integral identity for $p_{0}$ 
\begin{gather}
-\int \limits_{Q^{T}}\partial _{t}p_{0}\varphi {dxdt}+\int
\limits_{Q^{T}}\nabla p_{0}\nabla \varphi {dxdt}  \notag \\
+\mathcal{A}_{n}\int \limits_{\omega ^{T}}(p_{0}-\mathcal{B}_{n}H^{\ast
}(p_{0}))\varphi {dxdt}+\mathcal{A}_{n}\int \limits_{(\Omega \setminus 
\overline{\omega })\times (0,T)}(p_{0}-\mathcal{B}_{n}M^{\ast
}(p_{0}))\varphi {dxdt}  \notag \\
=\int \limits_{Q^{T}}\nabla u_{0}\nabla \varphi {dxdt}+\mathcal{A}_{n}\int
\limits_{Q^{T}}(u_{0}-\mathcal{B}_{n}(\mathcal{B}_{n}+N^{-1})G^{\ast
}(H(u_{0}))\varphi {dxdt}  \notag \\
+\mathcal{A}_{n}\int \limits_{(\Omega \setminus \overline{\omega })\times
(0,T)}(u_{0}-\mathcal{B}_{n}^{2}M^{\ast }(M(u_{0})))\varphi {dxdt},
\label{p_0 int ident}
\end{gather}%
that is valid for an arbitrary function $\varphi \in
L^{2}(0,T;H_{0}^{1}(\Omega ))$. Thus, the limit adjoint problem is 
\begin{equation}
\left \{ 
\begin{array}{ll}
-\partial _{t}p_{0}-\Delta p_{0}+ &  \\ 
+\mathcal{A}_{n}(p_{0}-\mathcal{B}_{n}H^{\ast }(p_{0}))\chi _{\omega ^{T}}+%
\mathcal{A}_{n}(p_{0}-\mathcal{B}_{n}M^{\ast }(p_{0}))\chi _{(\Omega
\setminus \overline{\omega })\times (0,T)}= &  \\ 
=-\Delta u_{0}+\mathcal{A}_{n}(u_{0}-\mathcal{B}_{n}(\mathcal{B}%
_{n}+N^{-1})G^{\ast }(H(u_{0}))\chi _{\omega ^{T}}+ &  \\ 
+\mathcal{A}_{n}(u_{0}-\mathcal{B}_{n}^{2}M^{\ast }(M(u_{0})))\chi _{(\Omega
\setminus \overline{\omega })\times (0,T)}, & (x,t)\in Q^{T} \\ 
p(x,T)=u_{0}(x,T), & x\in \Omega , \\ 
p(x,t)=0, & (x,t)\in \Gamma ^{T}.%
\end{array}%
\right.  \label{p_0 limit prob}
\end{equation}%
This concludes the proof of the homogenization theorem.
\end{proof}

\section{Proof of the limit cost functional and controls convergence theorems%
}

In this Section we will complete the characterization of the limit optimal
control.

\begin{proof}
of Theorem \ref{thm: cost func limit}. We start with the expression of $%
J_{\varepsilon }$ taken at $v_{\varepsilon }=-N^{-1}p_{\varepsilon }\chi
_{S_{\varepsilon }^{2,T}}$: 
\begin{gather*}
2J_{\varepsilon }(v_{\varepsilon })=\int \limits_{Q_{\varepsilon
}^{T}}|\nabla u_{\varepsilon }|^{2}{dxdt}+\int \limits_{\Omega _{\varepsilon
}}|u_{\varepsilon }(x,T)|^{2}{dx} \\
+\varepsilon ^{-\gamma }\int \limits_{S_{\varepsilon }}|u_{\varepsilon
}(x,T)|^{2}{ds}+N^{-1}\varepsilon ^{-\gamma }\int \limits_{S_{\varepsilon
}^{2,T}}p_{\varepsilon }^{2}{dsdt}.
\end{gather*}%
Using integral identity \eqref{int ident u}, we get 
\begin{gather}
2J_{\varepsilon }(v_{\varepsilon })=\int \limits_{Q_{\varepsilon
}^{T}}|\nabla u_{\varepsilon }|^{2}{dxdt}+\int \limits_{\Omega _{\varepsilon
}}|u_{\varepsilon }(x,T)|^{2}{dx}+\varepsilon ^{-\gamma }\int
\limits_{S_{\varepsilon }}|u_{\varepsilon }(x,T)|^{2}{ds}+\int
\limits_{Q_{\varepsilon }^{T}}fp_{\varepsilon }{dxdt}  \notag \\
-\int \limits_{Q_{\varepsilon }^{T}}\nabla u_{\varepsilon }\nabla
p_{\varepsilon }{dxdt}-\int \limits_{Q_{\varepsilon }^{T}}\partial
_{t}u_{\varepsilon }p_{\varepsilon }{dxdt}-\varepsilon ^{-\gamma }\int
\limits_{S_{\varepsilon }^{T}}\partial _{t}u_{\varepsilon }p_{\varepsilon }{%
dsdt}.  \label{thm: cost func step 1}
\end{gather}%
From the integral identity \eqref{int ident init adjoint prob}, we have 
\begin{gather*}
-\int \limits_{Q_{\varepsilon }^{T}}\nabla u_{\varepsilon }\nabla
p_{\varepsilon }{dxdt}-\int \limits_{Q_{\varepsilon }^{T}}\partial
_{t}u_{\varepsilon }p_{\varepsilon }{dxdt}-\varepsilon ^{-\gamma }\int
\limits_{S_{\varepsilon }^{T}}\partial _{t}u_{\varepsilon }p_{\varepsilon }{%
dsdt} \\
=-\int \limits_{Q_{\varepsilon }^{T}}\nabla p_{\varepsilon }\nabla
u_{\varepsilon }{dxdt}+\int \limits_{Q_{\varepsilon }^{T}}\partial
_{t}p_{\varepsilon }u_{\varepsilon }{dxdt}+\varepsilon ^{-\gamma }\int
\limits_{S_{\varepsilon }^{T}}\partial _{t}p_{\varepsilon }u_{\varepsilon }{%
dsdt} \\
-\varepsilon ^{-\gamma }\int \limits_{S_{\varepsilon }}|u_{\varepsilon
}(x,T)|^{2}{ds}-\int \limits_{\Omega_{\varepsilon }}|u_{\varepsilon
}(x,T)|^{2}{dx} \\
=-\int \limits_{Q_{\varepsilon }^{T}}|\nabla u_{\varepsilon }|^{2}{dxdt}%
-\varepsilon ^{-\gamma }\int \limits_{S_{\varepsilon }}|u_{\varepsilon
}(x,T)|^{2}{ds}-\int \limits_{\Omega_{\varepsilon }}|u_{\varepsilon
}(x,T)|^{2}{dx}.
\end{gather*}%
Substituting this equality into \eqref{thm: cost func step 1}, we get 
\begin{equation*}
2J_{\varepsilon }(v_{\varepsilon })=\int \limits_{Q_{\varepsilon
}^{T}}fp_{\varepsilon }{dxdt}.
\end{equation*}%
Then, passing to the limit as $\varepsilon \rightarrow 0$ we have 
\begin{equation*}
\lim \limits_{\varepsilon \rightarrow 0}2J_{\varepsilon }(v_{\varepsilon
})=\int \limits_{Q^{T}}fp_{0}{dxdt}.
\end{equation*}%
Using the integral identity \eqref{u_0 int ident} for the function $u_{0}$,
we get 
\begin{gather*}
\int \limits_{Q^{T}}fp_{0}{dxdt}=\int \limits_{Q^{T}}\partial _{t}u_{0}p_{0}{%
dxdt}+\int \limits_{Q^{T}}\nabla u_{0}\nabla p_{0}{dxdt}+\mathcal{A}_{n}\int
\limits_{\omega ^{T}}(u_{0}-\mathcal{B}_{n}H(u_{0}))p_{0}{dxdt} \\
+\mathcal{A}_{n}\int \limits_{(\Omega \setminus \overline{\omega })\times
(0,T)}(u_{0}-\mathcal{B}_{n}M(u_{0}))p_{0}{dxdt}+N^{-1}\mathcal{A}_{n}%
\mathcal{B}_{n}\int \limits_{\omega ^{T}}H(G^{\ast }(p_{0}))p_{0}{dxdt}= \\
=-\int \limits_{Q^{T}}\partial _{t}p_{0}u_{0}{dxdt}+\int \limits_{\Omega
}|u_{0}(x,T)|^{2}{dx}+\int \limits_{Q^{T}}\nabla p_{0}\nabla u_{0}{dxdt} \\
+\mathcal{A}_{n}\int \limits_{\omega ^{T}}(p_{0}-\mathcal{B}_{n}H^{\ast
}(p_{0}))u_{0}{dxdt}+\mathcal{A}_{n}\int \limits_{(\Omega \setminus 
\overline{\omega })\times (0,T)}(p_{0}-\mathcal{B}_{n}M^{\ast }(p_{0}))u_{0}{%
dxdt} \\
+N^{-1}\mathcal{A}_{n}\mathcal{B}_{n}\int \limits_{\omega ^{T}}H(G^{\ast
}(p_{0}))p_{0}{dxdt}.
\end{gather*}%
Then, we use the integral identity \eqref{p_0 int ident} for the function $%
p_{0}$ and duality relation for the operators $H$ and $M$, to conclude that 
\begin{gather*}
\int \limits_{Q^{T}}fp_{0}{dxdt}=\int \limits_{\Omega }|u_{0}(x,T)|^{2}{dxdt}%
+\int \limits_{Q^{T}}|\nabla u_{0}|^{2}{dxdt} \\
+\mathcal{A}_{n}\int \limits_{\omega ^{T}}(u_{0}-\mathcal{B}_{n}(\mathcal{B}%
_{n}+N^{-1})G^{\ast }(H(u_{0})))u_{0}{dxdt} \\
+\mathcal{A}_{n}\int \limits_{(\Omega \setminus \overline{\omega })\times
(0,T)}(u_{0}-\mathcal{B}_{n}^{2}M^{\ast }(M(u_{0})))u_{0}{dxdt}+N^{-1}%
\mathcal{A}_{n}\mathcal{B}_{n}\int \limits_{\omega ^{T}}H(G^{\ast
}(p_{0}))p_{0}{dxdt}.
\end{gather*}%
Let us show that the derived expression is non-negative. According to the
definition of $M$ and $M^{\ast }$, we have 
\begin{gather*}
\int \limits_{(\Omega \setminus \overline{\omega })\times (0,T)}(u_{0}-%
\mathcal{B}_{n}^{2}M^{\ast }(M(u_{0})))u_{0}{dxdt}=\int \limits_{(\Omega
\setminus \overline{\omega })\times (0,T)}(u_{0}^{2}-\mathcal{B}%
_{n}^{2}M^{2}(u_{0})){dxdt} \\
=\int \limits_{(\Omega \setminus \overline{\omega })\times (0,T)}((\partial
_{t}M(u_{0})+\mathcal{B}_{n}M(u_{0}))^{2}-\mathcal{B}_{n}^{2}M^{2}(u_{0})){%
dxdt} \\
=\int \limits_{(\Omega \setminus \overline{\omega })\times (0,T)}(\partial
_{t}M(u_{0}))^{2}{dxdt}+\mathcal{B}_{n}\int \limits_{\Omega \setminus 
\overline{\omega }}|M(u_{0})(x,T)|^{2}{dx} \\
=\int \limits_{(\Omega \setminus \overline{\omega })\times (0,T)}(u_{0}-%
\mathcal{B}_{n}M(u_{0}))^{2}{dxdt}+\mathcal{B}_{n}\int \limits_{\Omega
\setminus \overline{\omega }}|M(u_{0})(x,T)|^{2}{dx} \ge 0.
\end{gather*}

Using that $G^{\ast }(H(u_{0}))$ is the solution to the problem 
\eqref{second order
ODE} with boundary conditions \eqref{bound cond ODE}, we get 
\begin{gather*}
\int \limits_{\omega ^{T}}(u_{0}-\mathcal{B}_{n}(\mathcal{B}%
_{n}+N^{-1})G^{\ast }(H(u_{0})))u_{0}{dxdt} \\
=\int \limits_{\omega ^{T}}(u_{0}-\partial _{tt}^{2}G^{\ast
}(H(u_{0}))-u_{0})u_{0}{dxdt}=-\int \limits_{\omega ^{T}}\partial
_{tt}^{2}G^{\ast }(H(u_{0}))u_{0}{dxdt} \\
=\int \limits_{\omega ^{T}}\partial _{tt}^{2}G^{\ast }(H(u_{0}))(\partial
_{tt}^{2}G^{\ast }(H(u_{0}))-\mathcal{B}_{n}(\mathcal{B}_{n}+N^{-1})G^{\ast
}(H(u_{0}))){dxdt} \\
=\int \limits_{\omega ^{T}}(\partial _{tt}^{2}G^{\ast }(H(u_{0})))^{2}{dxdt}-%
\mathcal{B}_{n}(\mathcal{B}_{n}+N^{-1})\int \limits_{\omega ^{T}}\partial
_{tt}^{2}G^{\ast }(H(u_{0}))G^{\ast }(H({u_{0}})){dxdt} \\
=\int \limits_{\omega ^{T}}(\partial _{tt}^{2}G^{\ast }(H(u_{0})))^{2}{dxdt}+%
\mathcal{B}_{n}(\mathcal{B}_{n}+N^{-1})\int \limits_{\omega ^{T}}(\partial
_{t}G^{\ast }(H(u_{0})))^{2}{dxdt} \\
+\mathcal{B}_{n}\int \limits_{\omega }(\partial _{t}G^{\ast
}(H(u_{0}))(x,0))^{2}{dx} \ge 0.
\end{gather*}%
From here, we already see that the expression is non-zero, however, we
convert it to a form that is more similar to the one derived for the terms
with $M$. We have 
\begin{gather*}
\mathcal{B}_{n}\int \limits_{\omega }(\partial _{t}G^{\ast
}(H(u_{0}))(x,0))^{2}{dx} \\
=\mathcal{B}_{n}\int \limits_{\omega }((\partial _{t}G^{\ast
}(H(u_{0}))(x,0))^{2}-(\partial _{t}G^{\ast }(H(u_{0}))(x,T))^{2}){dx} \\
+\mathcal{B}_{n}\int \limits_{\omega }(\partial _{t}G^{\ast
}(H(u_{0}))(x,T))^{2}{dx} \\
=-\mathcal{B}_{n}\int \limits_{\omega ^{T}}\partial _{t}|\partial
_{t}G^{\ast }(H(u_{0}))|^{2}{dxdt}+\mathcal{B}_{n}\int \limits_{\omega
}(\partial _{t}G^{\ast }(H(u_{0}))(x,T))^{2}{dx} \\
=-2\mathcal{B}_{n}\int \limits_{\omega ^{T}}\partial _{tt}^{2}G^{\ast
}(H(u_{0}))\partial _{t}G^{\ast }(H(u_{0})){dxdt} +\mathcal{B}_{n}\int
\limits_{\omega }(\partial _{t}G^{\ast }(H(u_{0}))(x,T))^{2}{dx}.
\end{gather*}%
Note, that from the definition of $G^{\ast }(H(u_{0}))$, we have 
\begin{equation*}
\partial _{t}G^{\ast }(H(u_{0}))(x,T)=(\mathcal{B}_{n}+N^{-1})G^{\ast
}(H(u_{0}))(x,T)-H(u_{0})(x,T)=-H(u_{0})(x,T).
\end{equation*}%
Thus, we transform the last equality to 
\begin{gather*}
\mathcal{B}_{n}\int \limits_{\omega }(\partial _{t}G^{\ast
}(H(u_{0}))(x,0))^{2}{dx}= \\
=-2\mathcal{B}_{n}\int \limits_{\omega ^{T}}\partial _{tt}^{2}G^{\ast
}(H(u_{0}))\partial _{t}G^{\ast }(H(u_{0})){dxdt} + \mathcal{B}_{n}\int
\limits_{\omega }|H(u_{0})(x,T)|^{2}{dx}.
\end{gather*}%
Now, we put this into the original expression and get 
\begin{gather*}
\int \limits_{\omega ^{T}}(u_{0}-\mathcal{B}_{n}(\mathcal{B}%
_{n}+N^{-1})G^{\ast }(H(u_{0})))u_{0}{dxdt}= \\
=\int \limits_{\omega ^{T}}(\partial _{tt}^{2}G^{\ast }(H(u_{0})))^{2}{dxdt}+%
\mathcal{B}_{n}(\mathcal{B}_{n}+N^{-1})\int \limits_{\omega ^{T}}(\partial
_{t}G^{\ast }(H(u_{0})))^{2}{dxdt}- \\
-2\mathcal{B}_{n}\int \limits_{\omega ^{T}}\partial _{tt}^{2}G^{\ast
}(H(u_{0}))\partial _{t}G^{\ast }(H(u_{0})){dxdt} +\mathcal{B}_{n}\int
\limits_{\omega }|H(u_{0})(x,T)|^{2}{dx}.
\end{gather*}%
We transform it in the following way 
\begin{gather*}
\int \limits_{\omega ^{T}}(u_{0}-\mathcal{B}_{n}(\mathcal{B}%
_{n}+N^{-1})G^{\ast }(H(u_{0})))u_{0}{dxdt} \\
=\int \limits_{\omega ^{T}}(\mathcal{B}_{n}\partial _{t}G^{\ast
}(H(u_{0}))-\partial _{tt}^{2}G^{\ast }(H(u_{0})))^{2}{dxdt}+N^{-1}\mathcal{B%
}_{n}\int \limits_{\omega ^{T}}(\partial _{t}G^{\ast }(H(u_{0})))^{2}{dxdt}
\\
+\mathcal{B}_{n}\int \limits_{\omega }|H(u_{0})(x,T)|^{2}{dx}=\int
\limits_{\omega ^{T}}(\partial _{t}H(u_{0})-N^{-1}\partial _{t}G^{\ast
}(H(u_{0})))^{2}{dxdt} \\
+N^{-1}\mathcal{B}_{n}\int \limits_{\omega ^{T}}(\partial _{t}G^{\ast
}(H(u_{0})))^{2}{dxdt}+\mathcal{B}_{n}\int \limits_{\omega
}|H(u_{0})(x,T)|^{2}{dx}.
\end{gather*}
Using again the definition of operators $G^{\ast }$ and $H$, we have 
\begin{gather*}
\partial _{t}H(u_{0})-N^{-1}\partial _{t}G^{\ast }(H(u_{0}))=u_{0}-(\mathcal{%
B}_{n}+N^{-1})H(u_{0})+N^{-1}(\mathcal{B}_{n}+N^{-1})G^{\ast }(H(u_{0})) \\
+N^{-1}H(u_{0})-N^{-1}(\mathcal{B}_{n}+N^{-1})G^{\ast }(H(u_{0}))=u_{0}-%
\mathcal{B}_{n}H(u_{0}).
\end{gather*}%
Thus, we conclude 
\begin{gather*}
\int \limits_{\omega ^{T}}(u_{0}-\mathcal{B}_{n}(\mathcal{B}%
_{n}+N^{-1})G^{\ast }(H(u_{0})))u_{0}{dxdt} \\
=\int \limits_{\omega ^{T}}|u_{0}-\mathcal{B}_{n}H(u_{0})|^{2}{dxdt}+N^{-1}%
\mathcal{B}_{n}\int \limits_{\omega ^{T}}(\partial _{t}G^{\ast
}(H(u_{0})))^{2}{dxdt}+\mathcal{B}_{n}\int \limits_{\omega
}|H(u_{0})(x,T)|^{2}{dx}.
\end{gather*}

Using that $H(G^{\ast }(p_{0}))$ is a solution \eqref{second order ODE} we
get 
\begin{gather*}
\int \limits_{\omega ^{T}}H(G^{\ast }(p_{0}))p_{0}{dxdt}=-\int
\limits_{\omega ^{T}}H(G^{\ast }(p_{0}))(\partial _{tt}^{2}H(G^{\ast
}(p_{0}))-\mathcal{B}_{n}(\mathcal{B}_{n}+N^{-1})H(G^{\ast }(p_{0}))){dxdt}
\\
=\int \limits_{\omega ^{T}}(\partial _{t}H(G^{\ast }(p_{0})))^{2}{dxdt}-\int
\limits_{\omega }\partial _{t}(H(G^{\ast }(p_{0})))(x,T)H(G^{\ast
}(p_{0}))(x,T){dx} \\
+\mathcal{B}_{n}(\mathcal{B}_{n}+N^{-1})\int \limits_{\omega ^{T}}(H(G^{\ast
}(p_{0})))^{2}{dxdt}=\int \limits_{\omega ^{T}}(\partial _{t}H(G^{\ast
}(p_{0})))^{2}{dxdt} \\
+\mathcal{B}_{n}(\mathcal{B}_{n}+N^{-1})\int \limits_{\omega ^{T}}(H(G^{\ast
}(p_{0})))^{2}{dxdt}+(\mathcal{B}_{n}+N^{-1})\int \limits_{\omega
}(H(G^{\ast }(p_{0}))(x,T))^{2}{dx}.
\end{gather*}

Thus, we have 
\begin{gather*}
\lim \limits_{\varepsilon \rightarrow 0}2J_{\varepsilon }(v_{\varepsilon
})=\Vert \nabla u_{0}\Vert _{L^{2}(Q^{T})}^{2}+\Vert u_{0}(x,T)\Vert
_{L^{2}(\Omega )}^{2} \\
+\mathcal{A}_{n}\int \limits_{(\Omega \setminus \overline{\omega })\times
(0,T)}|u_{0}-\mathcal{B}_{n}M(u_{0})|^{2}{dxdt}+\mathcal{A}_{n}\mathcal{B}%
_{n}\int \limits_{\Omega \setminus \overline{\omega }}|M(u_{0})(x,T)|^{2}{dx}
\\
+\mathcal{A}_{n}\int \limits_{\omega ^{T}}|u_{0}-\mathcal{B}_{n}H(u_{0})|^{2}%
{dxdt}+\mathcal{A}_{n}\mathcal{B}_{n}\int \limits_{\omega
}|H(u_{0})(x,T)|^{2}{dx} \\
+N^{-1}\mathcal{A}_{n}\mathcal{B}_{n}\int \limits_{\omega ^{T}}(\partial
_{t}G^{\ast }(H(u_{0})))^{2}{dxdt}+N^{-1}\mathcal{A}_{n}\mathcal{B}_{n}\int
\limits_{\omega ^{T}}(\partial _{t}H(G^{\ast }(p_{0})))^{2}{dxdt} \\
+N^{-1}\mathcal{A}_{n}\mathcal{B}_{n}^{2}(\mathcal{B}_{n}+N^{-1})\int%
\limits_{\omega ^{T}}(H(G^{\ast }(p_{0})))^{2}{dxdt} \\
+N^{-1}\mathcal{A}_{n}\mathcal{B}_{n}(\mathcal{B}_{n}+N^{-1})\int
\limits_{\omega }(H(G^{\ast }(p_{0}))(x,T))^{2}{dx} =
J_0(-N^{-1}H(G^*(p_0))).
\end{gather*}

This concludes the proof.
\end{proof}

\bigskip

Finally, we will show that the system \eqref{homogenized system} is related
to the limit functional and the limit optimal control $v_{0}$.

\bigskip

\begin{proof}
of Theorem \ref{Thm convergence controls}. If $v_{0}$ is the optimal
control, then for any admissible function $v\in U_{ad}$ we have 
\begin{equation*}
J_{0}^{\prime }(v_{0})v = \lim \limits_{\lambda \rightarrow 0}\frac{%
J_{0}(v_{0}+\lambda v)-J_{0}(v_{0})}{\lambda }=0.
\end{equation*}%
We set $\theta =(u_{0}(v_{0}+\lambda v)-u_{0}(v_{0}))/\lambda $. The
function $\theta $ is a solution to the following problem 
\begin{equation}
\left \{ 
\begin{array}{ll}
\partial _{t}\theta -\Delta \theta +\mathcal{A}_{n}(\theta -\mathcal{B}%
_{n}H(\theta ))\chi _{\omega ^{T}} &  \\ 
+\mathcal{A}_{n}(\theta -\mathcal{B}_{n}M(\theta ))\chi _{(\Omega \setminus 
\overline{\omega })\times (0,T)}=\mathcal{A}_{n}\mathcal{B}_{n}v\chi
_{\omega ^{T}}, & (x,t)\in Q^{T}, \\ 
\theta (x,0)=0, & x\in \Omega , \\ 
\theta (x,t)=0, & (x,t)\in \Gamma ^{T},%
\end{array}%
\right.  \label{limit theta prob}
\end{equation}%
note that due to the linearity of the operator $M$ and $H$, we have 
\begin{gather*}
(H(u_0(v_0 + \lambda v) - H(u_0(v_0)))/\lambda = H(\theta), \\
(M(u_0(v_0 + \lambda v) - M(u_0(v_0)))/\lambda = M(\theta).
\end{gather*}

Thus, we have 
\begin{gather}
J_{0}^{\prime }(v_{0})v=\int \limits_{Q^{T}}\nabla \theta \nabla u_{0}(v_{0})%
{dxdt}+\int \limits_{\Omega }\theta (x,T)u_{0}(v_{0})(x,T){dx}  \notag \\
+\mathcal{A}_{n}\int \limits_{(\Omega \setminus \overline{\omega })\times
(0,T)}(\theta -\mathcal{B}_{n}M(\theta ))(u_{0}(v_{0})-\mathcal{B}%
_{n}M(u_{0}(v_{0}))){dxdt}  \notag \\
+\mathcal{A}_{n}\mathcal{B}_{n}\int \limits_{\Omega \setminus \overline{%
\omega }}M(\theta )(x,T)M(u_{0}(v_{0}))(x,T){dx}  \notag \\
+\mathcal{A}_{n}\int \limits_{\omega ^{T}}(\theta -\mathcal{B}_{n}H(\theta
))(u_{0}(v_{0})-\mathcal{B}_{n}H(u_{0}(v_{0}))){dxdt}  \notag \\
+\mathcal{A}_{n}\mathcal{B}_{n}\int \limits_{\omega }H(\theta
)(x,T)H(u_{0}(v_{0}))(x,T){dx}  \notag \\
+N^{-1}\mathcal{A}_{n}\mathcal{B}_{n}\int \limits_{\omega ^{T}}\partial
_{t}G(H^{\ast }(\theta ))\partial _{t}G(H^{\ast }(u_{0}(v_{0}))){dxdt}+N%
\mathcal{A}_{n}\mathcal{B}_{n}\int \limits_{\omega ^{T}}\partial
_{t}v_{0}\partial _{t}v{dxdt}  \notag \\
+N\mathcal{A}_{n}\mathcal{B}_{n}^{2}(\mathcal{B}_{n}+N^{-1})\int
\limits_{\omega ^{T}}v_{0}v{dxdt}+N\mathcal{A}_{n}\mathcal{B}_{n}(\mathcal{B}%
_{n}+N^{-1})\int \limits_{\omega }v_{0}(x,T)v(x,T){dx}.
\label{limit J deriv}
\end{gather}%
Now, we use that $p_{0}$ is a solution to the problem \eqref{p_0 limit
prob}, that is adjoint to \eqref{limit state problem func}, and get 
\begin{gather*}
\int \limits_{Q^{T}}\nabla \theta \nabla u_{0}(v_{0}){dxdt}=-\int
\limits_{Q^{T}}\partial _{t}p_{0}\theta {dxdt}+\int \limits_{Q^{T}}\nabla
p_{0}\nabla \theta {dxdt} \\
+\mathcal{A}_{n}\int \limits_{\omega ^{T}}(p_{0}-\mathcal{B}_{n}H^{\ast
}(p_{0}))\theta {dxdt}+\mathcal{A}_{n}\int \limits_{(\Omega \setminus 
\overline{\omega })\times (0,T)}(p_{0}-\mathcal{B}_{n}M^{\ast
}(p_{0}))\theta {dxdt} \\
-\mathcal{A}_{n}\int \limits_{\omega ^{T}}(u_{0}-\mathcal{B}_{n}(\mathcal{B}%
_{n}+N^{-1})G^{\ast }(H(u_{0})))\theta {dxdt} \\
-\mathcal{A}_{n}\int \limits_{(\Omega \setminus \overline{\omega })\times
(0,T)}(u_{0}-\mathcal{B}_{n}^{2}M^{\ast }(M(u_{0})))\theta {dxdt}.
\end{gather*}%
As $\theta $ is a solution to the problem \eqref{limit theta prob}, we have 
\begin{gather}
\int \limits_{Q^{T}}\nabla \theta \nabla u_{0}(v_{0}){dxdt}=-\int
\limits_{\Omega }u_{0}(v_{0})(x,T)\theta (x,T){dx}+\mathcal{A}_{n}\mathcal{B}%
_{n}\int \limits_{\omega ^{T}}vp_{0}{dxdt}  \notag \\
-\mathcal{A}_{n}\int \limits_{\omega ^{T}}(u_{0}-\mathcal{B}_{n}(\mathcal{B}%
_{n}+N^{-1})G^{\ast }(H(u_{0}))\theta {dxdt}-\mathcal{A}_{n}\int
\limits_{(\Omega \setminus \overline{\omega })\times (0,T)}(u_{0}-\mathcal{B}%
_{n}^{2}M^{\ast }(M(u_{0})))\theta {dxdt.}  \label{nabla theta rel}
\end{gather}%
Using the same transformations as in the proof of the theorem above, we have 
\begin{gather}
\int \limits_{(\Omega \setminus \overline{\omega })\times
(0,T)}(u_{0}(v_{0})-\mathcal{B}_{n}^{2}M^{\ast }(M(u_{0}(v_{0})))\theta {dxdt%
}  \notag \\
=\int \limits_{(\Omega \setminus \overline{\omega })\times
(0,T)}(u_{0}(v_{0})\theta -\mathcal{B}_{n}^{2}M(u_{0}(v_{0})))M(\theta ){dxdt%
}  \notag \\
=\int \limits_{(\Omega \setminus \overline{\omega })\times (0,T)}((\partial
_{t}M(u_{0}(v_{0}))+\mathcal{B}_{n}M(u_{0}(v_{0})))((\partial _{t}M(\theta )
\notag \\
+\mathcal{B}_{n}M(\theta ))-\mathcal{B}_{n}^{2}M(u_{0}(v_{0}))M(\theta ){dxdt%
}  \notag \\
=\int \limits_{(\Omega \setminus \overline{\omega })\times
(0,T)}(u_{0}(v_{0})-\mathcal{B}_{n}M(u_{0}(v_{0})))(\theta -\mathcal{B}%
_{n}M(\theta )){dxdt}  \notag \\
+\mathcal{B}_{n}\int \limits_{\Omega }M(u_{0}(v_{0}))(x,T)M(\theta )(x,T){dx}%
.  \label{theta and u rel 2}
\end{gather}%
And, similarly as in the proof of the theorem above, we get 
\begin{gather*}
\int \limits_{\omega ^{T}}(u_{0}(v_0)-\mathcal{B}_{n}(\mathcal{B}%
_{n}+N^{-1})G^{\ast }(H(u_{0}(v_0))))\theta {dxdt} \\
=\int \limits_{\omega ^{T}}(u_{0}(v_{0})-\mathcal{B}_{n}H(u_{0}(v_{0})))(%
\theta -\mathcal{B}_{n}H(\theta )){dxdt} \\
+\mathcal{B}_{n}\int \limits_{\omega }H(u_{0}(v_{0}))(x,T)H(\theta )(x,T){dx}
\\
+N^{-1}\mathcal{B}_{n}\int \limits_{\omega ^{T}}\partial _{t}G(H^{\ast
}(u_{0}(v_{0})))\partial _{t}G(H^{\ast }(\theta )){dxdt}.
\end{gather*}%
Indeed, we use that $G^{\ast }(H(u_{0}))$ is the solution to the problem 
\eqref{second order
ODE} with boundary conditions \eqref{bound cond ODE}, and derive 
\begin{gather*}
\int \limits_{\omega ^{T}}(u_{0}-\mathcal{B}_{n}(\mathcal{B}%
_{n}+N^{-1})G^{\ast }(H(u_{0})))\theta{dxdt} \\
=\int \limits_{\omega ^{T}}\partial _{tt}^{2}G^{\ast }(H(u_{0}))(\partial
_{tt}^{2}G^{\ast }(H(\theta))-\mathcal{B}_{n}(\mathcal{B}_{n}+N^{-1})G^{\ast
}(H(\theta))){dxdt} \\
=\int \limits_{\omega ^{T}}\partial _{tt}^{2}G^{\ast }(H(u_{0}))
\partial_{tt}^{2}G^{\ast }(H(\theta)){dxdt} - \mathcal{B}_{n}(\mathcal{B}%
_{n}+N^{-1})\int \limits_{\omega ^{T}}\partial _{tt}^{2}G^{\ast
}(H(u_{0}))G^{\ast }(H(\theta)){dxdt} \\
=\int \limits_{\omega ^{T}}\partial _{tt}^{2}G^{\ast }(H(u_{0}))
\partial_{tt}^{2}G^{\ast }(H(\theta)){dxdt} + \mathcal{B}_{n}(\mathcal{B}%
_{n}+N^{-1})\int \limits_{\omega ^{T}}\partial _{t}G^{\ast
}(H(u_{0}))\partial _{t}G^{\ast }(H(\theta)){dxdt} \\
+ \mathcal{B}_{n}\int \limits_{\omega }\partial _{t}G^{\ast
}(H(u_{0}))(x,0)\partial _{t}G^{\ast }(H(\theta))(x,0){dx}.
\end{gather*}

We transform the last term in the following way 
\begin{gather*}
\mathcal{B}_{n}\int \limits_{\omega }\partial _{t}G^{\ast
}(H(u_{0}))(x,0)\partial _{t}G^{\ast }(H(\theta))(x,0){dx} \\
=-\mathcal{B}_{n}\int \limits_{\omega ^{T}}\partial _{t}(\partial
_{t}G^{\ast }(H(u_{0}))G^{\ast }(H(\theta))){dxdt} \\
+ \mathcal{B}_{n}\int \limits_{\omega }\partial _{t}G^{\ast
}(H(u_{0}))(x,T)\partial _{t}G^{\ast }(H(\theta))(x,T){dx} \\
=-\mathcal{B}_{n}\int \limits_{\omega ^{T}}(\partial _{tt}^{2}G^{\ast
}(H(u_{0}))\partial _{t}G^{\ast }(H(\theta)) + \partial _{t} G^{\ast
}(H(u_{0}))\partial _{tt}^{2}G^{\ast }(H(\theta))){dxdt} \\
+ \mathcal{B}_{n}\int \limits_{\omega }\partial _{t}G^{\ast
}(H(u_{0}))(x,T)\partial _{t}G^{\ast }(H(\theta))(x,T){dx}.
\end{gather*}%
Note, that from the definition of $G^{\ast }(H(u_{0}))$, we have 
\begin{equation*}
\partial _{t}G^{\ast }(H(u_{0}))(x,T) = -H(u_{0})(x,T),\quad
\partial_{t}G^{\ast}(H(\theta))(x, T) = - H(\theta)(x, T).
\end{equation*}%
Thus, we transform the last equality to 
\begin{gather*}
\mathcal{B}_{n}\int \limits_{\omega }\partial _{t}G^{\ast
}(H(u_{0}))(x,0)\partial _{t}G^{\ast }(H(\theta))(x,0){dx} \\
= -\mathcal{B}_{n}\int \limits_{\omega ^{T}}(\partial _{tt}^{2}G^{\ast
}(H(u_{0}))\partial _{t}G^{\ast }(H(\theta)) + \partial _{t} G^{\ast
}(H(u_{0}))\partial _{tt}^{2}G^{\ast }(H(\theta))){dxdt} \\
+ \mathcal{B}_{n}\int \limits_{\omega }H(u_{0})(x,T) H(\theta)(x,T){dx}.
\end{gather*}%
Now, we put this into the original expression and get 
\begin{gather*}
\int \limits_{\omega ^{T}}(u_{0}-\mathcal{B}_{n}(\mathcal{B}%
_{n}+N^{-1})G^{\ast }(H(u_{0})))\theta {dxdt}= \\
=\int \limits_{\omega ^{T}}\partial _{tt}^{2}G^{\ast }(H(u_{0})) \partial
_{tt}^{2}G^{\ast }(H(\theta)){dxdt}+\mathcal{B}_{n}(\mathcal{B}%
_{n}+N^{-1})\int \limits_{\omega ^{T}}\partial _{t}G^{\ast }(H(u_{0}))
\partial _{t}G^{\ast }(H(\theta)){dxdt}- \\
- \mathcal{B}_{n}\int \limits_{\omega ^{T}}(\partial _{tt}^{2}G^{\ast
}(H(u_{0}))\partial _{t}G^{\ast }(H(\theta)) + \partial _{t} G^{\ast
}(H(u_{0}))\partial _{tt}^{2}G^{\ast }(H(\theta))){dxdt} \\
+\mathcal{B}_{n}\int \limits_{\omega }H(u_{0})(x,T) H(\theta)(x,T){dx}.
\end{gather*}%
We transform it in the following way 
\begin{gather*}
\int \limits_{\omega ^{T}}(u_{0}-\mathcal{B}_{n}(\mathcal{B}%
_{n}+N^{-1})G^{\ast }(H(u_{0})))\theta {dxdt} \\
= \int \limits_{\omega ^{T}}(\mathcal{B}_{n}\partial _{t}G^{\ast
}(H(u_{0}))-\partial _{tt}^{2}G^{\ast }(H(u_{0})))(\mathcal{B}_{n}\partial
_{t}G^{\ast }(H(\theta))-\partial _{tt}^{2}G^{\ast }(H(\theta))){dxdt} \\
+ N^{-1}\mathcal{B}_{n}\int \limits_{\omega ^{T}}\partial _{t}G^{\ast
}(H(u_{0})) \partial _{t}G^{\ast }(H(\theta)){dxdt} +\mathcal{B}_{n}\int
\limits_{\omega }H(u_{0})(x,T) H(\theta)(x,T){dx} \\
= \int \limits_{\omega ^{T}}(\partial _{t}H(u_{0})-N^{-1}\partial
_{t}G^{\ast }(H(\theta))) (\partial _{t}H(\theta)-N^{-1}\partial _{t}G^{\ast
}(H(\theta))){dxdt} \\
+ N^{-1}\mathcal{B}_{n}\int \limits_{\omega ^{T}}\partial _{t}G^{\ast
}(H(u_{0})) \partial _{t}G^{\ast }(H(\theta)){dxdt} +\mathcal{B}_{n}\int
\limits_{\omega }H(u_{0})(x,T) H(\theta)(x,T){dx}.
\end{gather*}
Using again the definition of operators $G^{\ast }$ and $H$, we have 
\begin{gather*}
\partial _{t}H(u_{0})-N^{-1}\partial _{t}G^{\ast }(H(u_{0}))=u_{0}-\mathcal{B%
}_{n}H(u_{0}), \\
\partial _{t}H(\theta)-N^{-1}\partial _{t}G^{\ast }(H(\theta))=u_{0}-%
\mathcal{B}_{n}H(\theta).
\end{gather*}%
Thus, we derive 
\begin{gather}
\int \limits_{\omega ^{T}}(u_{0}-\mathcal{B}_{n}(\mathcal{B}%
_{n}+N^{-1})G^{\ast }(H(u_{0})))\theta{dxdt}  \notag \\
=\int \limits_{\omega ^{T}}(u_{0}-\mathcal{B}_{n}H(u_{0}))(\theta - \mathcal{%
B}_{n}H(\theta)){dxdt} +  \notag \\
+ N^{-1}\mathcal{B}_{n}\int \limits_{\omega ^{T}}\partial _{t}G^{\ast
}(H(u_{0})) \partial _{t}G^{\ast }(H(\theta)){dxdt} +\mathcal{B}_{n}\int
\limits_{\omega }H(u_{0})(x,T) H(\theta)(x,T){dx}.  \label{theta and u rel 1}
\end{gather}

Substituting relations \eqref{theta and u rel 2} and \eqref{theta and u rel
1} into \eqref{nabla theta rel}, and then, putting it into \eqref{limit J
deriv}, we derive 
\begin{gather*}
J_{0}^{\prime }(v_{0})v=\mathcal{A}_{n}\mathcal{B}_{n}\int \limits_{\omega
^{T}}p_{0}v{dxdt}+N\mathcal{A}_{n}\mathcal{B}_{n}\int \limits_{\omega
^{T}}\partial _{t}v_{0}\partial _{t}v{dxdt} \\
+N\mathcal{A}_{n}\mathcal{B}_{n}^{2}(\mathcal{B}_{n}+N^{-1})\int
\limits_{\omega ^{T}}v_{0}v{dxdt}+N\mathcal{A}_{n}\mathcal{B}_{n}(\mathcal{B}%
_{n}+N^{-1})\int \limits_{\omega }v_{0}(x,T)v(x,T){dx}.
\end{gather*}%
Integrating by parts in the second integral, we get 
\begin{gather*}
J_{0}^{\prime }(v_{0})v=\mathcal{A}_{n}\mathcal{B}_{n}\int \limits_{\omega
^{T}}p_{0}v{dxdt}+N\mathcal{A}_{n}\mathcal{B}_{n}\int \limits_{\omega
}(\partial _{t}v_{0}(x,T)+(\mathcal{B}_{n}+N^{-1})v_{0}(x,T))v(x,T){dx} \\
+N\mathcal{A}_{n}\mathcal{B}_{n}\int \limits_{\omega ^{T}}(-\partial
_{tt}^{2}v_{0}+\mathcal{B}_{n}(\mathcal{B}_{n}+N^{-1})v_{0})v{dxdt}=0,
\end{gather*}%
for an arbitrary admissible function $v\in U_{ad}$. Hence, we get that $%
v_{0} $ must be a solution of the following problem 
\begin{equation*}
\left \{ 
\begin{array}{ll}
\partial _{tt}^{2}v_{0}-\mathcal{B}_{n}(\mathcal{B}%
_{n}+N^{-1})v_{0}=N^{-1}p_{0} &  \\ 
v_{0}(x,0)=0,\quad \partial _{t}v_{0}(x,T)+(\mathcal{B}%
_{n}+N^{-1})v_{0}(x,T)=0. & 
\end{array}%
\right.
\end{equation*}%
But, this is related to the problem satisfied by $H(G^{\ast }(-N^{-1}p_{0}))$%
. Due to the uniqueness of the solution of the problem we conclude that $%
v_{0}=-N^{-1}H(G^{\ast }(p_{0}))\chi _{\omega ^{T}}$.
\end{proof}

\bigskip

\begin{remark}
As in the papers \cite{DiPoShAttouch} and \cite{DiPoShRACSAM}, thanks to
Remark \ref{Rem target nonzero}, by simplifying the cost functional $%
J_{\varepsilon }$ as in \cite{DiPoShRACSAM}, by making the parameter $%
N\rightarrow 0$, it seems possible to show the approximate controllability
with final observation of solutions of the limit problem.
\end{remark}

\textbf{Acknowledgement. }The research of J.I. D\'{\i}az was partially
supported by the project PID2023-146754NB-I00 funded by
MCIU/AEI/10.13039/501100011033 and FEDER,\newline
EU.MCIU/AEI/10.13039/-501100011033/FEDER, EU.

\end{document}